\documentclass{article}
\RequirePackage{amsthm,amsmath,amsfonts,amssymb}
\usepackage{float}

\usepackage[utf8]{inputenc}
\usepackage{bbm}
\usepackage{bm} 
\usepackage{dsfont}
\usepackage{natbib}
\usepackage{paralist}
\usepackage{enumitem}
\usepackage[usenames]{color}
\usepackage{url}
\usepackage[colorlinks,
linkcolor=red,
anchorcolor=blue,
citecolor=blue
]{hyperref}

\usepackage[margin=1.5in]{geometry}  
\usepackage{xspace,exscale,relsize}
\usepackage{fancybox,shadow}
\usepackage{graphicx}
\usepackage{caption}
\usepackage{subcaption} 
\usepackage{xcolor}
\usepackage[title]{appendix}
\usepackage{caption}
\usepackage{extarrows}
\usepackage{comment}
\usepackage{booktabs} 
\usepackage{makecell, multirow}

\usepackage{mathrsfs}
\makeatletter
\let\over=\@@over \let\overwithdelims=\@@overwithdelims
\let\atop=\@@atop \let\atopwithdelims=\@@atopwithdelims
\let\above=\@@above \let\abovewithdelims=\@@abovewithdelims
\makeatother
\usepackage{rotating}

	\usepackage{ifpdf}
	
	\usepackage{psfrag}
	
	\usepackage{prettyref}

	\usepackage{tikz}
	\usetikzlibrary{arrows}
	\tikzstyle{int}=[draw, fill=blue!20, minimum size=2em]
	\tikzstyle{dot}=[circle, draw, fill=blue!20, minimum size=2em]
	\tikzstyle{init} = [pin edge={to-,thin,black}]

	\usepackage[all]{xy}

	\usepackage{algorithm}
	\usepackage{algpseudocode}
	
	\usepackage{mathtools}

   \def\bA{\boldsymbol A}  
   \def\bB{\boldsymbol B}

\def\be{{\boldsymbol e}}   \def\bE{\boldsymbol E}  \def\EE{\mathbb{E}}
\def\bff{{\boldsymbol f}}

   \def\bI{\boldsymbol I}

   \def\bM{\boldsymbol M}

     \def\PP{\mathbb{P}}
     
     \def\RR{\mathbb{R}}

\def\bu{{\boldsymbol u}}   \def\bU{\boldsymbol U}  
\def\bv{\boldsymbol v}   \def\bV{\boldsymbol V}  
\def\bw{\boldsymbol w}     
\def\bx{\boldsymbol x}   \def\bX{\boldsymbol X}  
\def\by{\boldsymbol y}     
\def\bz{\boldsymbol z}   \def\bZ{\boldsymbol Z}  

\def\11{\mathbbm{1}}

\def\calN{{\cal  N}}

\newcommand{\bfsym}[1]{\ensuremath{\boldsymbol{#1}}}

\def\balpha{\bfsym \alpha}

           \def\bDelta {\bfsym {\Delta}}
              
\def\bmu{\bfsym {\mu}}                 

\def\btheta{\bfsym {\theta}}           
          \def\bepsilon{\bfsym \varepsilon}
             \def\bSigma{\bfsym \Sigma}
        \def\bLambda {\bfsym {\Lambda}}

\def\bxi{\bfsym {\xi}}

\DeclareMathOperator{\cov}{cov}

\def\rF{{\mathrm{F}}}

\def\%#1\%{\begin{align}#1\end{align}}
\def\$#1\${\begin{align*}#1\end{align*}}

\graphicspath {{fig/}}

\usepackage{ifthen}
\newboolean{aos}
\setboolean{aos}{FALSE}

\ifx\eqref\undefined
\newcommand{\eqref}[1]{~(\ref{#1})}
\fi
\ifx\mod\undefined
\def\mod{\mathop{\rm mod}}
\fi

\usepackage{bm}

\def\exp{\mathop{\rm exp}}

\def\cov{\mathop{\rm cov}}

\def\EE{\Expect}
\newcommand{\sign}{{\sf {sign}}}

\def\PP{\mathbb{P}}

\def\simiid{\stackrel{iid}{\sim}}

\newcommand{\abs}[1]{\left| #1 \right|}

\makeatletter
\def\bbordermatrix#1{\begingroup \m@th
	\@tempdima 4.75\p@
	\setbox\z@\vbox{%
		\def\cr{\crcr\noalign{\kern2\p@\global\let\cr\endline}}%
		\ialign{$##$\hfil\kern2\p@\kern\@tempdima&\thinspace\hfil$##$\hfil
			&&\quad\hfil$##$\hfil\crcr
			\omit\strut\hfil\crcr\noalign{\kern-\baselineskip}%
			#1\crcr\omit\strut\cr}}%
	\setbox\tw@\vbox{\unvcopy\z@\global\setbox\@ne\lastbox}%
	\setbox\tw@\hbox{\unhbox\@ne\unskip\global\setbox\@ne\lastbox}%
	\setbox\tw@\hbox{$\kern\wd\@ne\kern-\@tempdima\left[\kern-\wd\@ne
		\global\setbox\@ne\vbox{\box\@ne\kern2\p@}%
		\vcenter{\kern-\ht\@ne\unvbox\z@\kern-\baselineskip}\,\right]$}%
	\null\;\vbox{\kern\ht\@ne\box\tw@}\endgroup}
\makeatother

\newcommand{\stepa}[1]{\overset{\rm (a)}{#1}}
\newcommand{\stepb}[1]{\overset{\rm (b)}{#1}}
\newcommand{\stepc}[1]{\overset{\rm (c)}{#1}}
\newcommand{\stepd}[1]{\overset{\rm (d)}{#1}}

\newcommand{\subG}{\mathsf{SG}}

\newcommand{\reals}{\mathbb{R}}
\newcommand{\naturals}{\mathbb{N}}

\newcommand{\Expect}{\mathbb{E}}

\newcommand{\iid}{iid\xspace}

\newcommand{\pth}[1]{\left( #1 \right)}
\newcommand{\qth}[1]{\left[ #1 \right]}
\newcommand{\sth}[1]{\left\{ #1 \right\}}

\definecolor{myblue}{rgb}{.8, .8, 1}
\definecolor{mathblue}{rgb}{0.2472, 0.24, 0.6} 
\definecolor{mathred}{rgb}{0.6, 0.24, 0.442893}
\definecolor{mathyellow}{rgb}{0.6, 0.547014, 0.24}

\newcommand{\sfS}{{\mathsf{S}}}
\newcommand{\sfT}{{\mathsf{T}}}

\newrefformat{eq}{(\ref{#1})}
\newrefformat{thm}{Theorem~\ref{#1}}
\newrefformat{th}{Theorem~\ref{#1}}
\newrefformat{chap}{Chapter~\ref{#1}}
\newrefformat{sec}{Section~\ref{#1}}
\newrefformat{seca}{Section~\ref{#1}}
\newrefformat{algo}{Algorithm~\ref{#1}}
\newrefformat{asmp}{Assumption~\ref{#1}}
\newrefformat{fig}{Fig.~\ref{#1}}
\newrefformat{tab}{Table~\ref{#1}}
\newrefformat{rmk}{Remark~\ref{#1}}
\newrefformat{clm}{Claim~\ref{#1}}
\newrefformat{def}{Definition~\ref{#1}}
\newrefformat{cor}{Corollary~\ref{#1}}
\newrefformat{lmm}{Lemma~\ref{#1}}
\newrefformat{prop}{Proposition~\ref{#1}}
\newrefformat{pr}{Proposition~\ref{#1}}
\newrefformat{property}{Property~\ref{#1}}
\newrefformat{app}{Appendix~\ref{#1}}
\newrefformat{apx}{Appendix~\ref{#1}}
\newrefformat{ex}{Example~\ref{#1}}
\newrefformat{exer}{Exercise~\ref{#1}}
\newrefformat{soln}{Solution~\ref{#1}}

\def\unifto{\mathop{{\mskip 3mu plus 2mu minus 1mu%
			\setbox0=\hbox{$\mathchar"3221$}%
			\raise.6ex\copy0\kern-\wd0%
			\lower0.5ex\hbox{$\mathchar"3221$}}\mskip 3mu plus 2mu minus 1mu}}

\ifx\lesssim\undefined
\def\simleq{{{\mskip 3mu plus 2mu minus 1mu%
			\setbox0=\hbox{$\mathchar"013C$}%
			\raise.2ex\copy0\kern-\wd0%
			\lower0.9ex\hbox{$\mathchar"0218$}}\mskip 3mu plus 2mu minus 1mu}}
\else
\def\simleq{\lesssim}
\fi

\ifx\gtrsim\undefined
\def\simgeq{{{\mskip 3mu plus 2mu minus 1mu%
			\setbox0=\hbox{$\mathchar"013E$}%
			\raise.2ex\copy0\kern-\wd0%
			\lower0.9ex\hbox{$\mathchar"0218$}}\mskip 3mu plus 2mu minus 1mu}}
\else
\def\simgeq{\gtrsim}
\fi

	\theoremstyle{definition}
	
	\newif\ifmapx
	{\catcode`/=0 \catcode`\\=12/gdef/mkillslash\#1{#1}}
	\edef\jobnametmp{\expandafter\string\csname ic_apx\endcsname}
	\edef\jobnameapx{\expandafter\mkillslash\jobnametmp}
	\edef\jobnameexpand{\jobname}
	\ifx\jobnameexpand\jobnameapx
	\mapxtrue
	\else
	\mapxfalse
	\fi

	\renewcommand{\hat}{\widehat}
	\renewcommand{\tilde}{\widetilde}

	\newcommand{\snr}{{\sf{SNR}}}

	\newcommand{\dist}{{\sf dist}}

	\newtheorem{theorem}{Theorem}
	\newtheorem{lemma}[theorem]{Lemma}
	\newtheorem{corollary}[theorem]{Corollary}

	\theoremstyle{definition}
	
	\newtheorem{remark}{Remark}

	\newtheorem{myassumption}{Assumption}

\makeatletter
\let\over=\@@over \let\overwithdelims=\@@overwithdelims
\let\atop=\@@atop \let\atopwithdelims=\@@atopwithdelims
\let\above=\@@above \let\abovewithdelims=\@@abovewithdelims
\makeatother
\begin{document}

\title{Factor Adjusted Spectral Clustering for Mixture Models}

\date{}

\author{Shange Tang, Soham Jana and Jianqing Fan\thanks{
		S.T. and J.F. are with the Department of Operations Research and Financial Engineering, Princeton University, Princeton, NJ, email: \url{shangetang@princeton.edu} and \url{jqfan@princeton.edu}.
		S.J. is with the Department of ACMS, University of Notre Dame, Notre Dame,
		IN, email: \url{sjana2@nd.edu}.}}


\maketitle

\begin{abstract}
This paper studies a factor modeling-based approach for clustering high-dimensional data generated from a mixture of strongly correlated variables. Statistical modeling with correlated structures pervades modern applications in economics, finance, genomics, wireless sensing, etc., with factor modeling being one of the popular techniques for explaining the common dependence. Standard techniques for clustering high-dimensional data, e.g., naive spectral clustering, often fail to yield insightful results as their performances heavily depend on the mixture components having a weakly correlated structure. To address the clustering problem in the presence of a latent factor model, we propose the Factor Adjusted Spectral Clustering (FASC) algorithm, which uses an additional data denoising step via eliminating the factor component to cope with the data dependency. We prove this method achieves an exponentially low mislabeling rate, with respect to the signal to noise ratio under a general set of assumptions. Our assumption bridges many classical factor models in the literature, such as the pervasive factor model, the weak factor model, and the sparse factor model. The FASC algorithm is also computationally efficient, requiring only near-linear sample complexity with respect to the data dimension. We also show the applicability of the FASC algorithm with real data experiments and numerical studies, and establish that FASC provides significant results in many cases where traditional spectral clustering fails.
\end{abstract}

\noindent%

\section{Introduction}

\subsection{Overview}

Clustering is an important unsupervised problem in statistics and machine learning that aims to partition data into two or more homogeneous groups. Applications exist in diverse areas, ranging over biology \citep{herwig1999large}, finance \citep{cai2016clustering}, wireless sensing \citep{mamalis2009clustering}, etc. In many modern applications, we encounter high-dimensional data, which can significantly increase the computational complexity of standard clustering algorithms such as $K$-means \citep{kumar2004ASL}. In addition, a stylized feature of high-dimensional data is a common dependence among features that make clustering algorithms statistically inefficient.  These make it essential to produce clustering algorithms that incorporate a dimension reduction-based dependent-adjustment step before performing the clustering step.

Spectral clustering \citep{von2007tutorial,bach2003learning,ng2001spectral} has been at the forefront of such dimension-reduction-based methods due to their efficiency. Spectral methods attempt to utilize a low-dimensional data-driven projection strategy, which can preserve the cluster membership information. For example, \cite{loffler2021optimality} attempts to utilize the properties of projecting the data along the subspace containing the span of the actual cluster centers. The spectral clustering methods provably work well in the presence of sub-Gaussian data \citep{loffler2021optimality,abbe2022ellp}. However, the performance guarantees of the vanilla spectral clustering techniques deteriorate significantly when the underlying distributions have ill-conditioned covariance matrices \citep{davis2021clustering} or highly correlated structures \citep{li2020cast}. This is reasonable as the directions of variations in the data set can be generated from both the directional components of the covariance matrix and the subspace spanned by the cluster centroids, and the standard spectral methods are incapable of differentiating them.

High-dimensional measurements are often highly related and admit a factor structure that result in an ill-conditioned covariance matrix.  Compared to numerous results for clustering spherical data distributions (see \cite{lu2016statistical,jana2024general} and the references therein), few papers in the literature address the challenges of clustering with dependent and ill-conditioned covariance structures. In general, analyzing highly correlated data is a stylized feature in many applications, and the dependency structure adversely affects standard statistical tasks, such as model selection \cite{fan2020factor} and multiple testing \citep{doi:10.1080/01621459.2018.1527700}. Factor modeling \citep{bai2002determining, fan2020statistical} is one of the widely used strategies for analyzing such data; important examples exist in macroeconomics \citep{bai2002determining}, finance \citep{fama2015five}, computational biology \citep{fan2014challenges}, social science \citep{zhou2015measuring}, etc.  It decomposes the measurements into common parts that depend on latent factors through a loading matrix and idiosyncratic components that are weakly correlated.  The number of latent factors that drive dependence among measurements is usually small.

This paper proposes the {\it Factor Adjusted Spectral Clustering (FASC)} algorithm that guarantees remarkable data clustering performance. Our method first obtains dependence-adjusted data  before applying a spectral clustering method. This preliminary step ensures that the transformed data have approximately isotropic covariance, enabling the spectral step to guarantee better performance. We will formally show that our algorithm can guarantee a significantly small clustering error under very general conditions on the data dimensions and the factor loading matrices. Notably, our assumptions allow covariances to be significantly ill-conditioned, where the vanilla spectral method might fail. Simulation studies and real data expositions lend further support to the claimed performance improvement.

To gain the insights of our method,  let us consider the following toy example of the Gaussian mixture model in $d$-dimensions: 
\begin{align}\label{eq:model0}
	\PP(y_i=1)=\PP(y_i=-1)=\frac{1}{2} \quad \mbox{and} \quad \bx_i | y_i \sim \mathcal{N}(y_i \bmu, \bSigma),\quad
	\bmu\in \reals^d,\bSigma\in \reals^{d\times d},
\end{align}
with the goal to recover the unobserved class label $y_i \in \{\pm 1\}$ for each observation $\bx_i$.  
Given $\boldsymbol{\mu}$ and $\bSigma$, the optimal classifier is ${\sign}(\langle  \bx_{i}, \bSigma^{-1} \boldsymbol{\mu} \rangle)$, where $\sign(a)={a\over |a|}$.    The misclassication rate is $\Phi(-\sqrt{{\snr}})$,  where the signal-to-noise ratio is  ${\snr}=\boldsymbol{\mu}^{\sfT} \bSigma^{-1} \boldsymbol{\mu}$. However, when the variables are highly correlated, as the following numerical experiment exhibits, this optimal mislabeling rate is not always achievable by spectral clustering. Consider model  \eqref{eq:model0} with 
\begin{align*}
	\begin{gathered}
		n=200, \quad
		d=100,  \quad
		\bmu = (10,0,\cdots, 0)^{\sfT},\quad
		\bSigma = t\bB\bB^\sfT+  \bI_d,
		t\in \sth{0, 1,\dots,100}\\
		\bB \in \reals^{100\times 3},\quad B_{ij}\simiid \calN(0,1).
	\end{gathered}
\end{align*}	
For each given $t$, we implement the spectral clustering algorithm (\prettyref{algo:spectral} \citep{abbe2022ellp,zhang2022leaveoneout}) on $\{\bx_i\}_{i=1}^n$ to detect the two clusters.
We plot the mislabeling rate of spectral clustering and the optimal rate $\Phi(-\sqrt{{\snr}})$ as a function of $t$ in \prettyref{fig:spectral clustering for different correlation}.
\begin{figure}
	\centering
	\includegraphics[width=\linewidth]{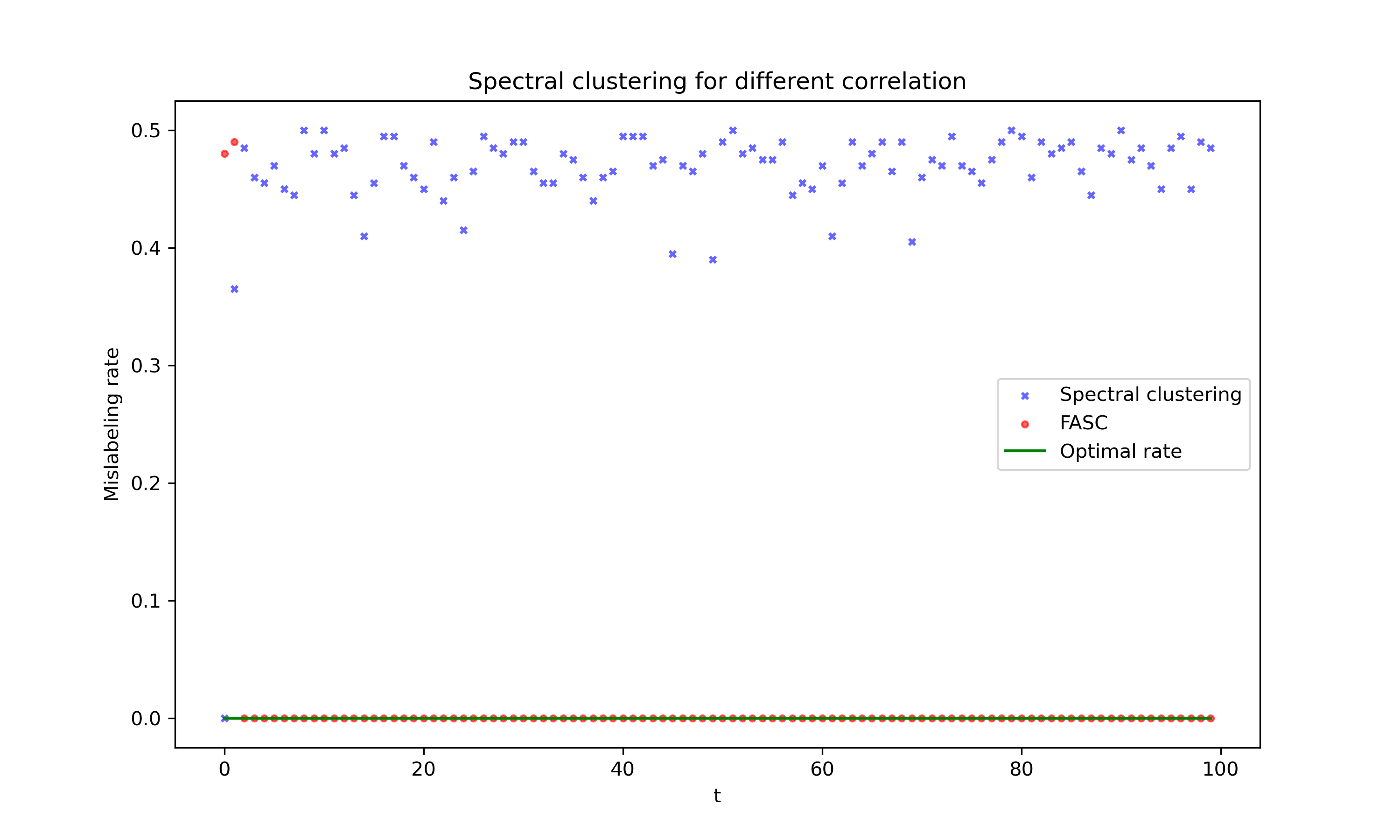}
	\caption{Comparison of mislabeling errors for FASC, vanilla spectral clustering, and the optimal one for  $t=1,\dots,100$. }
	\label{fig:spectral clustering for different correlation}
\end{figure}
The result reveals that the vanilla spectral clustering algorithm's performance deteriorates as the covariance matrix deviates from the isometric structure.
An explanation is as follows. It is known that the vanilla spectral clustering algorithm \citep{abbe2022ellp,zhang2022leaveoneout} can achieve a mislabeling rate of $e^{-\Omega({\|\bmu\|_{2}^{2}}/{\|\bSigma\|_{2}})}$, which might be significantly smaller than the optimal rate $\Phi(-\sqrt{{\snr}})\approx e^{-\Omega({\snr})}$.  When $t$ increases, the {\snr} remains high, but $\|\bSigma\|_2$ can be quite large, resulting in a small 
$ {\|\bmu\|_{2}^{2}}/{\|\bSigma\|_{2}}$, which can explain the poor mislabeling rate of the vanilla spectral clustering, as \prettyref{fig:spectral clustering for different correlation} shows. In comparison, the above figure shows that the mislabeling obtained by the FASC algorithm is close to the optimal performance.

Our proposed method, FASC, has many advantages. Firstly, FASC applies to a wide range of scenarios, including 
\begin{itemize}
	\item the classical factor model with pervasiveness \citep{fan2021robust}, where the singular values of the factor loading matrix scales as to the square root of data dimension,
	\item the weak factor model \citep{onatski2012asymptotics,bryzgalova2015spurious,lettau2020estimating} where the singular values of the factor loading matrix are of constant order,
	\item the sparsity induced factor models \citep{uematsu2022estimation} where the singular values of the factor loading matrix can scale at different rates.
\end{itemize} 
Secondly, our clustering scheme works in high-dimensional setups with dimensions close to the sample size (up to logarithmic factors) and is computationally efficient. We also show that under a general data-generating model that encompasses all the above scenarios, the mislabeling rates produced by FASC are similar to the optimal mislabeling rate for clustering anisotopic Gaussian mixture models \citep{chen2021optimal}. Under the scenarios we consider, where the covariance matrices are ill-conditioned, previous works that achieve nice mislabeling guarantees are either computationally inefficient or require significantly smaller dimensional regimes than the one we consider, e.g., sample complexity of $n= \tilde{\Omega}(d^2)$ or even higher \citep{davis2021clustering,Brubaker2008IsotropicPA, Ge2015LearningMO, 10.1145/3519935.3519953}. Thirdly, although a few previous works considered clustering with an underlying factor structure \citep{longford1992factor,subedi2013clustering,viroli2010dimensionally}, none of them provide a theoretical guarantee for the proposed algorithms since the underlying factor structure makes the analysis of clustering much more complicated. In particular, the idea of factor adjustment in FASC to remove the factor component from the data has been explored before in \cite{fan2020factor} and \cite{doi:10.1080/01621459.2018.1527700} in the context of model selection and multiple testing. Compared with the above, the clustering task is significantly challenging as multiple data-generating clusters make it difficult to differentiate the directional components in the data due to factors and the cluster locations. We provide guarantees for our method FASC by carefully analyzing the relationship between the factors and the idiosyncratic components while preserving the information on cluster memberships.

The rest of the paper is organized as follows. We describe the problem setup of
the mixture model with factor structure in  \prettyref{sec:setup}. \prettyref{sec:method} introduces the FASC algorithm.
The theoretical guarantees of the FASC algorithm is described in \prettyref{sec:theory}. 
We provide numerical examples and real data-based studies in \prettyref{sec:simulation} and \prettyref{sec:real_data}. 
The proof sketch for our main theoretical results is provided in \prettyref{sec:proof_sketch}. The details of the proof and related technical results are included in \prettyref{app:proof_details} and \prettyref{app:technical_details}, respectively, in the supplemental material.

\subsection{Related works}

The study of mixture models has an illustrious history. The problem of data classification from a mixture model dates back to the seminal work \cite{fisher1936use}. \cite{Friedman1989RegularizedDA} showed the necessity for the sample size $n$ to exceed the dimension $d$ to achieve the Bayes optimal rate of misclassification. In scenarios where $d$ surpasses $n$, \cite{bickel2004some} explained the superiority of the ``independence" rule over the Fisher linear discriminant rule. Additionally, \cite{fan2008high} contributed to this body of knowledge by shedding light on the challenges posed by noise accumulation in high-dimensional feature spaces when employing all features for classification. In response, they introduced FAIR, which selects a subset of critical features for high-dimensional classification, yielding commendable performance. When the data is generated from a Gaussian mixture model, another interesting line of work focuses on estimating the underlying mixture distribution, e.g., using EM algorithms \citep{wu2021randomly,balakrishnan2017statistical,pmlr-v65-daskalakis17b,kwon2020algorithm}, the moment methods \citep{ma2023best,doss2023optimal,liu2021settling,liu2022learning}, etc.

Clustering problems, particularly mislabeling minimization, are relatively less explored until recent developments. In the case of sub-Gaussian distributions, when the covariance matrix is known to be nearly spherical, i.e., $\bSigma=\bI_{d}$, many methods have been proposed, including iterative methods \citep{lu2016statistical,jana2023adversarially,jana2024general} and spectral methods \citep{VEMPALA2004841,jin2017phase,Ndaoud2018SharpOR,loffler2021optimality}. Of particular relevance to our investigation, \cite{loffler2021optimality} demonstrated that vanilla spectral clustering achieves the optimal mislabeling rate with a linear sample complexity $n= \Omega (d)$ in the context of Gaussian mixture models with spherical covariance matrices. In particular, they show that when the Gaussian components have centroids given by $\bmu_1,\dots,\bmu_K$ and each of the error coordinates has a variance $\sigma^2$, the optimal mislabeling rate is given by $e^{-\Omega({\snr})}$ with $\snr={\min_{i\neq j\in [K]}\|\bmu_i-\bmu_j\|_2^2\over \sigma^2}$. In scenarios where the covariance matrix $\bSigma$ is unknown, \cite{abbe2022ellp, zhang2022leaveoneout} showed spectral clustering achieves the mislabeling rate $e^{-\Omega(\sfS)}$, where $\sfS={\min_{i\neq j\in [K]}\|\bmu_i-\bmu_j\|^2\over \|\bSigma\|}$, with a (nearly) linear sample size $n = \tilde{\Omega}(d)$ for sub-Gaussian mixture models. However, we will show that the FASC algorithm achieves a much smaller mislabeling rate with an underlying factor structure.

The FASC algorithm we propose here draws inferences about the underlying factor structure using the top eigenvectors of the sample covariance matrix and requires the exact knowledge of the number of factors. This is closely related to the use of principal components analysis (PCA) \citep{RePEc:bes:jnlasa:v:97:y:2002:m:december:p:1167-1179} in factor analysis, which has emerged as the most prevalent technique in the literature. Multiple variants of PCA have also been introduced for factor estimation, including projected PCA \citep{10.1214/15-AOS1364}, diversified projection \citep{Fan2019LearningLF}, and others. Extensive exploration of PCA's asymptotic behavior under high-dimensional scenarios has been conducted in the literature, including \cite{johnstone2009consistency}, \cite{fan2013large}, \cite{shen2016statistics}, \cite{fan2015asymptotics}, and more. On the other hand, a long line of research has been dedicated to the estimation of the number of factors, with notable studies conducted by \cite{bai2002determining}, \cite{luo2009contour}, \cite{hallin2007determining}, \cite{lam2012factor}, and \cite{ahn2013eigenvalue},  \cite{fan2022estimating}, among others. In particular, \cite{fan2023factor} suggested that one can draw significant conclusions if the number of factors is known within a neighborhood of the true value; however, their strategy relies on using a neural network to approximate the underlying structure efficiently. Whether such a strategy is implementable for clustering purposes is beyond the scope of the current work.

\section{Problem setup}
\label{sec:setup}

\subsection{Notations} 
Let $[K]$ denote the set of integers $\{1,\dots,K\}$. Let $\|\cdot\| $ be the $\ell_2$ norm of a vector or the spectral norm of a matrix. We denote by $\|\cdot\|_{\rF}$ the Frobenius norm of a matrix. For a matrix $\bM$, we denote by $\sigma_{\min}(\bM)$ and $\sigma_{\max}(\bM)$ the smallest singular value and the largest singular value of $\bM$, respectively. For a symmetric matrix $\bA$, we denote by $\lambda_{\min}(\bA)$ and $\lambda_{\max}(\bA)$ the smallest eigenvalue and the largest eigenvalue of $\bA$, respectively. In addition, given any symmetric matrix $\bA\in \reals^{K\times K}$ we will denote its eigenvalues as $\lambda_{\max}(\bA):= \lambda_1(\bA)\geq \lambda_2(\bA)\geq\dots\geq \lambda_K(\bA)=\lambda_{\min}(\bA)$.
For a random variable $X\in \reals$, we say $X\in \subG(\sigma^2)$, where $\subG(\sigma^2)$ is the class of (one-dimensional) sub-Gaussian random variables with parameter $\sigma^{2}$, if $\mathbb{E}e^{t(X-\EE\qth{X})} \leq e^{\frac{1}{2}\sigma^{2}t^{2}}$ for any $t \in \mathbb{R}$. For a random vector $\bX \in \mathbb{R}^{d}$, we say $\bX\in \subG_{d}(\sigma^{2})$, where $\subG_d(\sigma^2)$ denotes the class of $d$-dimensional sub-Gaussian random variables with parameter $\sigma^{2}$, if $\bv^{\sfT}\bX \in \subG(\sigma^{2})$ holds for any unit vector $\bv \in \mathbb{R}^{d}$. We use $c,C_0,C_1,\dots$ to denote constants that may differ from line to line.



\subsection{Mixture model with factor structure} 
\label{section_factor}

We consider the additive noise model with $K$ clusters to describe the data $\bx_1,\dots,\bx_n$, 
\begin{align}
	\label{eq:factor_cluster}
	\bx_i\in \reals^d,
	\quad \bx_i=\bmu_{y_i}+\be_i,\quad y_i\in[K], \quad \be_i=\bB \bff_{i} + \bepsilon_i, \quad i\in [n],
\end{align}
where 
\begin{itemize}
	\item $\bmu_1,\dots,\bmu_K$ are the centroids corresponding to the $K$ clusters,
	\item for each $i\in [n]$, $y_i$ denotes the cluster label of data $\bx_i$,
	\item  $\bB\in \reals^{d\times r}$ is the factor loading matrix and $\bff_1,\dots,\bff_n\in \reals^r$ are the latent factors,
	\item the additive noises $\bepsilon_1,\dots,\bepsilon_n$ are independent mean zero random vectors in $\reals^d$.
\end{itemize}
The additive noise model is standard in the literature for modeling mixtures of sub-Gaussian distributions \citep{lu2016statistical,balakrishnan2017statistical,klusowski2016statistical}. We make the following distributional assumptions on the random quantities described above.
\begin{myassumption}\label{asmp:model}
	The factor and error variables are distributed as follows.
	\begin{enumerate}[label=(\alph*)]
		\item $\bff_i\in \subG_r(c_1^2)$ for some constant $c_1>0$ and  $\EE\qth{\bff_i}=0, \EE\qth{\bff_i \bff_i^\sfT}=\bI_r$,
		\item $\{y_i\}_{i=1}^n\simiid y$ with $\PP\qth{y=j}=p_j, j\in [K]$. We further assume that  there is a constant $c\in (0,1)$ such that $\min_{j\in [K]}p_j>\frac c{K}$ and the model is centered, i.e., $\sum_{j=1}^{K}p_{j}\boldsymbol{\mu}_{j} = 0$,
		\item $\bepsilon_i \in \subG_{d}(\sigma^2)$ where $0<\sigma<c'<1$ for a sufficiently small constant $c'>0$,
		\item $d\to \infty$ as $n\to \infty$ and $r,K\leq C$ for a large constant $C$.
	\end{enumerate}
\end{myassumption}

\begin{remark}
	We explain the above assumptions here. \prettyref{asmp:model}(a) is a standard regularity assumption in factor model literature. We assume the covariance of the factor $\bff_i$ is $\bI_r$ so that the model is identifiable (we can always rotate $\bB$ to satisfy this assumption). Also, we require the factor $\bff_i$ to be sub-Gaussian, which is also standard. \prettyref{asmp:model} (b) essentially centralize the cluster centers so that the mean of the data is zero, see \cite[Assumption 3.1]{abbe2022ellp} for similar conditions. Also, we do not want any of the cluster sizes to be too small so that the clusters degenerate,  therefore we require $\min_{j\in [K]}p_j>\frac c{K}$. \prettyref{asmp:model} (c) is a standard sub-Gaussian assumption for the noise. \prettyref{asmp:model} (d) describes the regime we consider, where the number of factors and the number of clusters are both at a constant level, which is common in factor model literature and cluster analysis literature. Also, we consider the scenario where both $n$ and $d$ go to infinity, which is common in modern high-dimensional data.
\end{remark}



\paragraph{Mislabeling rate}
The goal of this work is to construct an algorithm that produces estimates $\hat \by = \{\hat{y}_{i}\}_{i=1}^{n}$ of the true labels $\by = \{{y}_{i}\}_{i=1}^{n}$ to guarantee minimum mislabeling error. Notably the labels can only be learned up to a permutation ambiguity. Given the above, we define the minimum mislabeling as
\begin{equation*}
	\mathcal{M}(\hat{\by}, \by)=n^{-1} \min _{\tau \in S_K}\left|\left\{i \in[n]: \hat{y}_i \neq \tau\left(y_i\right)\right\}\right| .
\end{equation*}
where $S_K$ is the set of all permutations of $[K]$. Our goal is to design an efficient algorithm that obtains $\hat\by$ from the unlabeled data $\{\bx_{i}\}_{i=1}^{n}$ with a small $\mathcal{M}(\hat{\by}, \by)$.

\section{Method: Factor Adjusted Spectral Clustering (FASC)}
\label{sec:method}
For clarity of presentation, we first introduce the vanilla spectral clustering method in \prettyref{algo:spectral}. This algorithm can obtain a vanishing mislabeling error when the data is weakly correlated \citep{loffler2021optimality}.

\begin{algorithm}[H]
	\caption{Spectral clustering}\label{algo:spectral}
	\begin{algorithmic}[1]
		\State {\bf Input:} $\{\boldsymbol{x}_i\}^{n}_{i=1}$. The number of clusters $K$, target dimension of embedding $k\leq K$, error threshold $\epsilon>0$.
		
		\State Compute $
			\hat{\bV}:=(\bxi_{1},\cdots,\bxi_{k}),$
		where $\bxi_{1},\cdots,\bxi_{k}$ are the top $k$ right singular vectors of $(\bx_{1},\cdots,\bx_{n})^{T} $.
		
		\State Conduct $(1+\epsilon)$-approximate $K$-means \citep{kumar2004ASL} on $\{\hat{\bV}^{T}\bx_{i}\}_{i=1}^{n}$, getting $\{\hat{y}_{i}\}_{i=1}^{n}$ and $\{\hat{\boldsymbol{\mu}}_{j}\}_{j=1}^{K}$ such that:
		\begin{equation*}
			\sum_{i=1}^n\|\hat{\bV}^{T}\bx_{i}-\hat{\boldsymbol{\mu}}_{\hat{y}_i}\|_2^2 \leq(1+\epsilon) \min _{\substack{\{\widetilde{\boldsymbol{\mu}}_j\}_{j=1}^K \in  \mathbb{R}^k \\\left\{\widetilde{y}_i\right\}_{i=1}^n \in [K]}}\{\sum_{i=1}^n\|\hat{\bV}^{T}\bx_{i}-\widetilde{\boldsymbol{\mu}}_{\tilde{y}_i}\|_2^2\} .
		\end{equation*}

		\State {\bf Output:} $\{\hat{y}_{i}\}_{i=1}^{n}$.
	\end{algorithmic}
\end{algorithm}
To tackle the case with strong data dependency, we leverage ideas from factor analysis to propose the Factor Adjusted Spectral Clustering (FASC) algorithm (\prettyref{algo:fasc}). 
Under model \eqref{eq:factor_cluster}, we can write
\begin{align}\label{eq:ui-representation}
	\bx_{i}=\bB\bff_{i}+\bu_{i},\quad \bu_{i}=\bmu_{y_{i}} + \bepsilon_{i}.
\end{align}
Note that the idiosyncratic component $\bu_i$ essentially retains the cluster information $y_i$ of the data $\bx_i$ via $\bmu_{y_i}$. Under the assumption that the covariance structure of the idiosyncratic component is well-conditioned, an application of \prettyref{algo:spectral} on $\{\bu_i\}_{i=1}^n$ would have guaranteed label outputs with a small mislabeling error, as the variance component from $\bB \bff$ has been subtracted. Unfortunately, $\{\bu_i\}_{i=1}^n$ would be unknown in practice, and one of our contributions to the FASC algorithm is to estimate it. The approach of FASC is to first learn the latent factor components $\bB\bff_{i}$ by using principal component analysis, estimate the idiosyncratic components $\bu_{i}=\bx_i-\bB\bff_i$, and then apply the spectral clustering algorithm (\prettyref{algo:spectral}) on the estimated idiosyncratic components to obtain the final label estimates. By subtracting the latent factor components, we not only approximately decorrelate the variables but also reduce the noise in the data, enabling the spectral clustering method to achieve the optimal mislabeling rate. The formal description of the FASC algorithm is given below. For ease of proving theoretical results, later on, we incorporate a data splitting step. For simplicity of notations, we assume the sample size is $2n$.
\begin{algorithm}[H]
	\caption{Factor Adjusted Spectral Clustering (FASC)}\label{algo:fasc}
	\begin{algorithmic}[1]
		\State {\bf Input:} Dataset $\{\bx_i\}^{2n}_{i=1}$. The number of factors $r$, clusters $K$, and the target dimension of embedding $k\leq K$.
		\State Split the data $\{\bx_i\}^{2n}_{i=1}$ into two halves, and use the second half of data to compute $\hat{\bV}_{r}$: Let $\hat{\bSigma}:= \frac{1}{n}\sum_{i=n+1}^{2n} \bx_{i} \bx_{i}^\sfT$. Then we write the eigendecomposition of $\hat{\bSigma}$:
		\begin{equation*}
			\hat{\bSigma}=\sum_{j=1}^{d}\hat{\lambda}_{j}\hat{\bv}_{j}\hat{\bv}_{j}^\sfT,
		\end{equation*}
		where $\hat{\lambda}_{1} \geq \hat{\lambda}_{2} \geq \cdots \geq \hat{\lambda}_{d} \geq 0$. Let 
		\begin{equation} \label{hat_V_r}
			\hat{\bV}_{r}:= (\hat{\bv}_{1},\cdots,\hat{\bv}_{r}) \in \mathbb{R}^{d \times r}.
		\end{equation}
		\State Compute $\hat{\bu}_{i}$ for $i=1,\cdots,n$:
		\begin{equation} \label{eq:hat_u_i}
			\hat{\bu}_{i} = \bx_{i} - \hat{\bV}_{r}\hat{\bV}_{r}^\sfT\bx_{i}, \quad i=1,\cdots,n.
		\end{equation}
		
		\State Compute $\{\hat y_i\}_{i=1}^n$ as the output of applying \prettyref{algo:spectral} on $\{\hat{\bu}_{i}\}_{i=1}^n$ with projection dimension $k$ and number of clusters $K$.

		\State {\bf Output:} $\{\hat{y}_{i}\}_{i=1}^{n}$.
	\end{algorithmic}
\end{algorithm}


\begin{remark} In \prettyref{algo:fasc} we obtain the label estimates $\sth{\hat y_1,\dots,\hat y_n}$. By switching the role of $\sth{\bx_1,\dots,\bx_n}$ to estimate the factor components, we can similarly estimate the second half of clustering labels as $\sth{\hat y_{n+1},\dots,\hat y_{2n}}$. However, it is expected that there might be a permutation ambiguity of the estimated labels in the first and second half of the data. This can be easily resolved by matching the labels corresponding to clusters with similar centroid estimates in the first and second halves. This is justified because the label estimates via spectral clustering can lead to consistent estimation of the cluster centroids \cite[Proposition 3.1]{zhang2022leaveoneout}. However, the related technical details are beyond the scope of the current work.
\end{remark}

\begin{remark}
	The projection dimension $k$ required to guarantee desired theoretical properties of \prettyref{algo:spectral} (which extends to the guarantees of the FASC algorithm) depends on the rank of the subspace containing the cluster centers and might not be known beforehand. In such scenarios, one might want to choose $k=K$ as the number of clusters to be extracted is often known from practical experience or according to problem specifications. When the underlying data-generating distributions consist of an unknown number of mixture components, the problem of estimating $K$ is well-defined in the literature. For example, see \cite{hamerly2003learning,sugar2003finding} and the references therein for a detailed discussion. However, such analysis is beyond the scope of the current manuscript, and we will assume that we know the values of $k, K$. See the previous work of \cite{abbe2022ellp} for similar assumptions.
\end{remark}

\section{Theory}
\label{sec:theory}

\subsection{Main results}

For the rest of the paper, we will assume the following regularity conditions on the centroids $\bmu_1,\dots,\bmu_K$ of the underlying model 
\eqref{eq:factor_cluster}.
\begin{myassumption}[Regularity]\label{asmp:regularity}
	
	\begin{enumerate}[label=(\alph*)]
		\item \label{reg1} There is a constant $c_1$ such that $n\geq c_1d \log n$.	
		
		\item \label{reg2} $\mathbb{E}\qth{\boldsymbol{\mu}_{y_{i}}\boldsymbol{\mu}_{y_{i}}^\sfT}$ is rank $k$ and there exist constants $c_2, c_3>0$, such that $$c_2 \geq \lambda_{1} (\mathbb{E}\qth{\boldsymbol{\mu}_{y_{i}}\boldsymbol{\mu}_{y_{i}}^\sfT}) \geq \lambda_{k} (\mathbb{E}\qth{\boldsymbol{\mu}_{y_{i}}\boldsymbol{\mu}_{y_{i}}^\sfT}) \geq c_3.$$
		
		\item \label{reg3} There exist constants $c_4,c_5>0$ such that 
		$$c_5 \geq \max_{i\neq j}\|\boldsymbol{\mu}_{i}-\boldsymbol{\mu}_{j}\|\geq \min_{i\neq j}\|\boldsymbol{\mu}_{i}-\boldsymbol{\mu}_{j}\| \geq c_4.$$
		
	\end{enumerate}
\end{myassumption}

\begin{remark} \prettyref{asmp:regularity}\ref{reg2} is similar to \cite[Assumption 3.1]{abbe2022ellp} and ensures that the spectral projection step in \prettyref{algo:spectral} captures all the major directions in the space spanned by the mean vectors $\bmu_1,\dots,\bmu_K$. The inequality $\min_{i\neq j}\|\boldsymbol{\mu}_{i}-\boldsymbol{\mu}_{j}\| \geq c_4$ in \prettyref{asmp:regularity}\ref{reg3} is a standard separation criteria to guarantee small mislabeling errors. The condition $c_5 \geq \max_{i\neq j}\|\boldsymbol{\mu}_{i}-\boldsymbol{\mu}_{j}\|$ in \prettyref{asmp:regularity} (c) controls the variation in the data set along the space spanned by the centroid vectors to ensure that the variation due to the factor components can be separated easily. See  \cite{abbe2022ellp,zhang2022leaveoneout,loffler2021optimality} for a discussion on such comparable regularity conditions.
	
\end{remark}
Next, we assume a set of general conditions on the factor loading matrix $\bB$.
%

\begin{myassumption}[Factor loading matrix]
	\label{asmp:factors_genr}
	Let $n\geq C d(\log n)^3$ for a large constant $C>0$. Suppose that  $\mathbb{E}\qth{\boldsymbol{\mu}_{y_{i}}\boldsymbol{\mu}_{y_{i}}^\sfT}$ and $\bB\bB^\sfT$ obtain the following spectral decompositions
	\begin{align}
		\label{eq:eigen-decomp}
		\begin{gathered}
			\mathbb{E}\qth{\boldsymbol{\mu}_{y_{i}}\boldsymbol{\mu}_{y_{i}}^\sfT}=\bM \tilde \bLambda\bM^\sfT,\quad
			\bB\bB^\sfT=\bU\bLambda \bU^\sfT,
			\quad \tilde \bLambda,\bLambda\in\reals^{k\times k} \text{ are diagonal}.
		\end{gathered}
	\end{align}
	Then there are constants $\sth{\gamma_i}_{i=1}^4$ such that the following holds true.
	\begin{enumerate}[label=(\alph*)]
		\item \label{genr_factor1} $\sigma_{\min}(\bB)^2 \geq 3 (\|\EE \qth{\bmu_{y_i}\bmu_{y_i}^T} \|+ \sigma^2 )$ and $\sigma_{\max}(\bB) \geq 1$.
		
		\item \label{genr_factor2} $\frac{\sigma_{\max}(\bB)^2}{\sigma_{\min}(\bB)^2} (\sigma_{\max}(\bB) + \sqrt{d}) \leq \gamma_1 (\sigma \vee \frac{1}{\sqrt{\log n}}) \sqrt{\frac{n}{\log n}}$.
		
		\item \label{genr_perpendicularity} $\|\bU^{\sfT}\bM\| \leq \gamma_2$ and $\frac{\sigma_{\max}(\bB)}{\sigma_{\min}(\bB)^2} (\|\bU^T \bM\| + \sigma^2) \leq \gamma_3 (\sigma \vee \sqrt{\frac{1}{\log n}})$. If the noise variances $\sth{\EE\qth{\bepsilon_i\bepsilon_i^\sfT}}_{i=1}^n$ have a common isotropic structure, then the final assumption can be weakened to $\frac{\sigma_{\max}(\bB)}{\sigma_{\min}(\bB)^2} \|\bU^T \bM\| \leq \gamma_4 (\sigma \vee \sqrt{\frac{1}{\log n}})$.
	\end{enumerate}
\end{myassumption}

\begin{remark}
	\prettyref{asmp:factors_genr}\ref{genr_factor1} describes the scenario where the component of data variability due to the factors is much larger than the data variability due to the presence of different clusters. This enables us to filter out the factor component before applying the spectral method in \prettyref{algo:spectral}. \prettyref{asmp:factors_genr}\ref{genr_factor2} is a regularity condition that controls the discrepancy between the different directions of data deviations resulting from the underlying factors. \prettyref{asmp:factors_genr}\ref{genr_perpendicularity} ensures that the space of factors and the space of the centroids lie in nearly orthogonal directions so that their effects can be separated easily. Our assumptions are sufficient to guarantee a consistent estimation of clusters, and optimizing the assumptions is left for future work. However, in the later sections, we will see that these conditions are general enough to include various standard setups in factor modeling.
	
\end{remark}
With the above assumptions, we are now able to state our first main theoretical result. Denote the signal-to-noise ratio
\begin{align}
	\overline{\snr}:= \frac{\min_{i\neq j}\|\bmu_{i}-\bmu_{j}\|^{2}}{\sigma^2},
\end{align}
which we will use to describe the related mislabeling guarantees.
\begin{theorem}\label{thm:general}
	Consider the model (\ref{eq:factor_cluster}). Let $\hat{\by} := (\hat{y}_{1}, \cdots, \hat{y}_{n})$ be the output of \prettyref{algo:fasc}, $\by := (y_{1}, \cdots, y_{n})$ be the true labels.  Then, under  \prettyref{asmp:regularity} and \prettyref{asmp:factors_genr} ,  there exist constants $C,\bar C$ and $c$ such that 
	\begin{enumerate}[label=(\alph*)]
		\item If $\overline{\snr} > C \log n $, then $\lim _{n \rightarrow \infty} \mathbb{P}[\mathcal{M}(\hat{\by}, \by)=0]=1$.
		
		\item If $\bar C \leq \overline{\snr} \leq C \log n$, then $\mathbb{E} [\mathcal{M}(\hat{\by}, \by)] \leq \exp{(-c \cdot \overline{\snr})}$ for all sufficiently large $n$.
	\end{enumerate}
\end{theorem}
Theorem \ref{thm:general} shows that the FASC algorithm proposed in \prettyref{algo:fasc} achieves the exponentially small mislabeling rate $\exp{(-c \cdot \overline{\snr})}$ with near linear (with respect to the data dimension) sample complexity. The above mislabeling rate is significantly smaller than the existing error bounds in the literature that are attained by the vanilla spectral clustering method in \prettyref{algo:spectral}. For example, the mislabeling rates for the spectral clustering algorithm presented in \cite{abbe2022ellp,zhang2022leaveoneout} is given by $\exp{(- \Omega (\sfS))}$ where $\sfS := \frac{\min_{i\neq j}\|\bmu_{i}-\bmu_{j}\|^{2}}{\|\bB\bB^T + \sigma^2 \bI_d\|}$.  Note that $\|\bB\bB^\sfT\| \geq \max_{j \leq k} \|\bB_j\|^2$  is typically of order $d$, where $\bB_j$ is the $j^{th}$ column of $\bB$. Thus,
$\sfS$ is typically an order of magnitude smaller than $\overline \snr$.
This shows that our results are comparatively much stronger. The performance difference between FASC and spectral clustering will be further demonstrated by simulations in \prettyref{sec:simulation}. 

\subsection{Implications for factor model with the pervasiveness condition}
\label{sec:factorr_model_with_pervasiveness}

In this subsection, we assume the factor loading matrix $\bB$ satisfies the following ``pervasiveness" condition. 
\begin{myassumption}[Pervasiveness] \label{asmp:pervasiveness}
	$C_1d \leq \sigma_{\min}^2(\bB) \leq\sigma_{\max}^2(\bB) \leq C_2 d$ for constants $C_1,C_2>0$.  \end{myassumption}
This condition is common in the related literature (see \cite{fan2021robust} for discussions) regarding modeling purposes. For example, the condition is satisfied if we assume that each of the entries $B_{ij}$ are generated \iid from a $\subG(1)$ distribution. In addition, in high-dimensional scenarios, this assumption helps separate the factor component from the idiosyncratic component in the data, enabling a consistent learning of the factors. 
\begin{myassumption}[Approximate perpendicularity]\label{asmp:approximate_perpendicularity_pervasive}
There exists a small enough constant $c$ for which the eigen-decompositions of $\mathbb{E}\qth{\boldsymbol{\mu}_{y_{i}}\boldsymbol{\mu}_{y_{i}}^\sfT}$ and $\bB\bB^\sfT$ as defined in \eqref{eq:eigen-decomp} satisfy
	\begin{equation*} 
		\|\bU^\sfT\bM\| \leq \min \sth{\sqrt{d} \pth{\sigma \vee \sqrt{\frac{1}{\log n}}} , c}.
	\end{equation*}
\end{myassumption}

\begin{remark}
\prettyref{asmp:approximate_perpendicularity_pervasive} can be easily satisfied in the high-dimensional regime, which is the area of interest here. A simple example that satisfies the above assumption is that the space of the mean vectors and the space of the factor loadings are generated independently. In particular, consider the Grassmannian manifold $G_{d,k}$ that consists of all $k$-dimensional subspaces of $\mathbb{R}^{d}$ and let $\text{Unif} (G_{d,k})$ be the uniform distribution on the collection of randomly chosen $k$-dimensional subspace of $\mathbb{R}^{d}$ (see \cite{vershynin2018high} for a detailed description). Then we have the following result (see \prettyref{app:justify_approx_perp_lemma_proof} for a proof) that shows that  \prettyref{asmp:approximate_perpendicularity_pervasive} holds with a high probability whenever $d\geq C_1\log n$ for a large constant $C_1>0$.
\end{remark}
\begin{lemma} \label{lmm:justify_approx_perp_lemma}
Let $\bB \in \RR^{d\times r}$ be a fixed matrix, $\bU \in \RR^{d\times r}$ be the matrix of left singular vectors of $\bB$, and $\bM \sim \text{Unif} (G_{d,k})$. Then we have the following for some constants $c,\tilde c>0$.
\begin{itemize}
	\item If $d \geq c\log n$ then $\|\bU^\sfT\bM\| \leq \tilde c$ with a probability at least $1-n^{-9}$.
	\item If $d > c (\log n )^2$, then $\|\bU^\sfT\bM\| \leq \frac{\tilde c}{\sqrt{\log n}}$ with a probability at least $1-n^{-9}$.
\end{itemize}
\end{lemma}

As \prettyref{asmp:pervasiveness} and \prettyref{asmp:approximate_perpendicularity_pervasive} are specialized versions of \prettyref{asmp:factors_genr}, we can conclude the following.

\begin{corollary}\label{cor:pervasive}
	The outcomes of \prettyref{thm:general} hold under the pervasiveness condition in \prettyref{asmp:pervasiveness} and the approximate perpendicularity condition \prettyref{asmp:approximate_perpendicularity_pervasive}.
\end{corollary}

\subsection{Implications for models with weak factors}

Next, we describe the ``weak factor" model that satisfies the following condition.
\begin{myassumption}[Weak factor condition]
	\label{asmp:sigma_val_weak}
	There exists a constant $C$, such that $3(\|\EE \qth{\bmu_{y_i}\bmu_{y_i}^T} \|+ \sigma^2) \leq \sigma_{\min}^{2}(\bB) \leq \sigma_{\max}^{2}(\bB) \leq C$.
\end{myassumption}
The above modeling assumption includes factor models previously studied in \cite{onatski2012asymptotics,bryzgalova2015spurious,lettau2020estimating}. For example, \cite{lettau2020estimating} assumes that $\bB^\sfT \bB$ is approximately $\bI_r$, which implies that $\sigma_{\max}(\bB)$ is of constant order. The weak factor scenario makes it harder to separately detect the variation due to factors and cluster centroids, as both of them are now of constant order. As a result, we will need a slightly stronger ``approximate perpendicularity" condition.
The approximate perpendicularity assumption we use here can be easily satisfied in the high-dimensional regime whenever $d \geq \Omega((\log n)^2)$ (\prettyref{lmm:justify_approx_perp_lemma}).

\begin{myassumption}[Approximate perpendicularity for weak factor models]\label{asmp:approximate_perpendicularity_weak} Assume
\begin{equation*} 
	\|\bU^\sfT\bM\| \leq \min \sth{\sigma \vee \sqrt{\frac{1}{\log n}} , c}.
\end{equation*}
for some $c> 0$, where we adopt the notation of eigen-decomposition in \eqref{eq:eigen-decomp}.
\end{myassumption}

As the above assumptions satisfy \prettyref{asmp:factors_genr}, we  obtain the following theoretical guarantees for the FASC algorithm under the ``weak factor" scenario.
\begin{corollary}\label{cor:weak}
	The outcomes of \prettyref{thm:general} hold under the conditions in weak factor condition in \prettyref{asmp:sigma_val_weak} and the corresponding perpendicularity condition \prettyref{asmp:approximate_perpendicularity_weak}.
\end{corollary}

\subsection{Implications for models with disproportionate factor loadings}

Finally, we analyze the scenario where the smallest and the largest singular values of the factor loading matrix can scale with the dimension at a significantly different rate.
\begin{myassumption}[Disproportionate singular value condition] 
	\label{asmp:disproportionate_factors}
	There exists a constant $C$ such that $3 (\|\EE \qth{\bmu_{y_i}\bmu_{y_i}^T} \|+ \sigma^2 ) \leq \sigma_{\min}^{2}(\bB) \leq \sigma_{\max}^{2}(\bB) \leq C d$ and $\lim_{d\to \infty}{\sigma_{\min}(\bB)\over \sigma_{\max}(\bB)}\to 0$.
\end{myassumption}
The above modeling assumption includes the sparsity-induced factor model previously studied in the literature. For example, while studying the factor models in \cite{uematsu2022estimation}, the authors assume that the eigenvalues of $\bB\bB^\sfT$ satisfy $\lambda_{\ell}(\bB\bB^\sfT)=d^{\alpha_\ell}$ where $\alpha_1 >\dots>\alpha_r$ are constants in $(0,1)$. Under such a disproportionate eigenvalue scenario, an even stronger perpendicularity assumption is needed to ensure that the factors can be learned properly and the FASC algorithm can achieve the optimal mislabeling rate. This assumption can be satisfied easily by choosing the columns of the factor loading matrix from an orthogonal space with respect to the space of the true centroids. We note that it is beyond the scope of the current work to see whether this condition can be satisfied with a more classical data-generating setup, similar to the one mentioned in \prettyref{lmm:justify_approx_perp_lemma}.
\begin{myassumption}[Approximate perpendicularity]\label{asmp:approximate_perpendicularity_disproportionate}
Let the eigen-decomposition of $\mathbb{E}\qth{\boldsymbol{\mu}_{y_{i}}\boldsymbol{\mu}_{y_{i}}^\sfT}$ and $\bB\bB^\sfT$ be defined in \eqref{eq:eigen-decomp}. We assume the following holds for some constant $c$:
 \begin{equation*} 
     \|\bU^\sfT\bM\| \leq \min \sth{\frac{\sigma}{\sqrt{d}}, c}.
 \end{equation*}
\end{myassumption}
In view of the above assumptions, we have the following result.
\begin{corollary}\label{cor:dispro}
	Consider the model (\ref{eq:factor_cluster}) with $\EE\qth{\bepsilon_i\bepsilon_i^\sfT}= \sigma^2 \bI_d$ and let $n\geq \tilde C d^3(\log n)^3$ for a large constant $\tilde C>0$. Then the outcomes of \prettyref{thm:general} hold under the conditions in \prettyref{asmp:disproportionate_factors} and \prettyref{asmp:approximate_perpendicularity_pervasive}.
\end{corollary}
\prettyref{cor:dispro} shows that under this scenario, when a strong version of ``approximate perpendicularity" \prettyref{asmp:approximate_perpendicularity_disproportionate} holds, FASC can still achieve the optimal mislabeling rate. However, the sample complexity needed is $n=\tilde{\Omega} (d^3)$, which is much worse than the counterpart in \prettyref{cor:pervasive} and \prettyref{cor:weak}.

\section{Simulation studies}
\label{sec:simulation}

\subsection{Gaussian mixture models with correlated measurements}


In the following simulations, we generate samples using the model \eqref{eq:factor_cluster} with the number of clusters $K=5$, the number of factors $r=3$, and the data dimension $d$ is varied over the set $\{5,20,100,500\}$. For each such combination of $d,K,r$ we repeat the following data generation process 100 times, indexed by $t=1,2,\cdots,100$, with sample size $n=1000$ each
\begin{itemize}
	\item generate the factor
	loading matrix $\bB\in \mathbb{R}^{d \times r}$ whose rows are drawn from i.i.d. $\mathcal{N}(0,\bI_{r})$
	\item draw \iid vectors $\{\boldsymbol{\theta}_{j}\}_{j=1}^{K}$ from $\mathcal{N}(0,\frac{1}{d}\bI_{d})$ and denote $\boldsymbol{\mu}_{j} = \boldsymbol{\theta}_{j} - \frac{1}{K}\sum_{i=1}^{K}\boldsymbol{\theta}_{i}$
	\item obtain class label $\sth{y_i}_{i=1}^n \simiid \text{Uniform}([K])$, latent factors $\sth{\bff_{i}}_{i=1}^n \simiid \mathcal{N}(0,\bI_{r})$ and noise $\bepsilon_{i} \simiid \mathcal{N}(0,\sigma^{2}\bI_{d})$ with $\sigma= 0.01t$
	\item produce the observed data data set $\sth{\bx_i}_{i=1}^n$ with $\bx_i=\bmu_{y_i}+\bB \bff_i + \bepsilon_i$.
\end{itemize}
We then perform the FASC (\prettyref{algo:fasc}) and the vanilla spectral clustering (\prettyref{algo:spectral}) on these generated data and compare their mislabeling proportions. For simulation purposes, we use the entire data set to estimate the factors and perform the spectral step rather than using data splitting as required in \prettyref{algo:fasc}.

In \prettyref{fig:FASC VS Spectral clustering}, we plot the mislabeling rate of FASC (red points) and spectral clustering (blue points) with respect to $\sigma=\sth{0.01t}_{t=1}^{100}$ for different dimensions $d=5,20,100,500$ respectively. From \prettyref{fig:FASC VS Spectral clustering} it can be observed that vanilla spectral clustering fails for highly correlated data and the FASC algorithm achieves improved mislabeling errors regardless of the noise level.  The explanation is that the ``noise'' $\bB \bff_i$ gets suppressed or removed by the factor adjustment and hence increases the SNR.  In the extreme case, FASC achieves nearly exact recovery (mislabeling rate is close to $0$) when $\sigma$ is small. When $\sigma$ grows, the mislabeling rate of FASC grows until it reaches the level of spectral clustering.     When $t=1$, even with the noise from the common part $\bB \bff_i$ is removed, the signals are still too weak for good labeling.
The simulation results support \prettyref{thm:general}: when $\overline \snr$ exceeds some constant multiple of $\log n$, \prettyref{algo:fasc} is able to recover all the labels with high probability; when $\overline{\snr}$ is not that large, FASC can still recover the cluster labels significantly better than the spectral method, albeit with a non-zero mislabeling error.

\begin{figure}
    \centering
    \includegraphics[width=\linewidth]{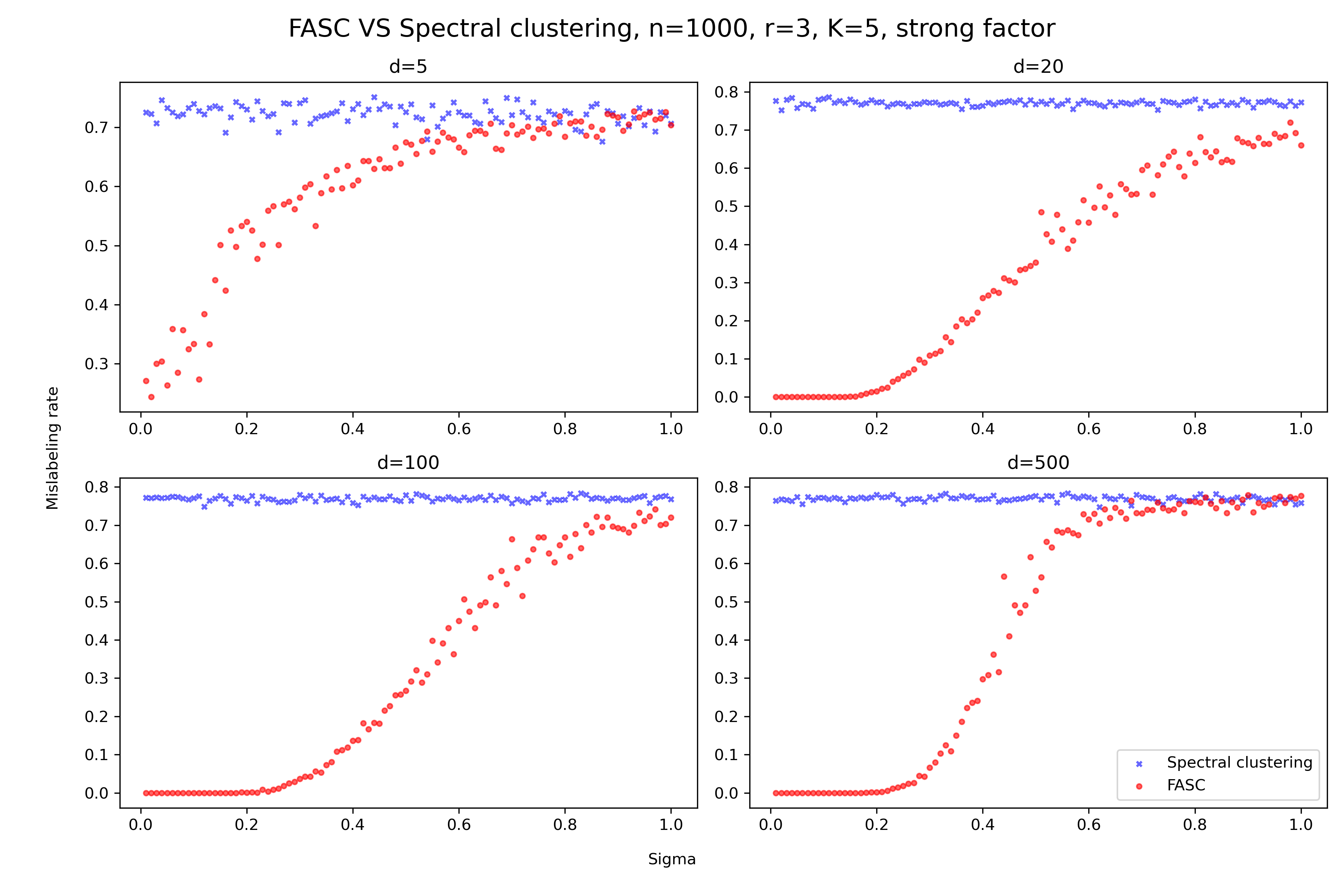}
    \caption{Mislabeling rates for FASC and vanilla spectral clustering with strong factors.}
    \label{fig:FASC VS Spectral clustering}
\end{figure}

\medskip

\begin{remark} We point out that for the cases $d=5$, FASC does not perform as well as it does when $d$ is larger. The approximate perpendicularity may not hold when $d$ is very small, and therefore the latent factors and its associated loading matrix can not be well estimated. This can be justified through further experiments: in the following experiments, we perform spectral clustering on $\bu_i$. In this case, spectral clustering should provide optimal mislabeling. We compare the results of FASC on $\bx_{i}$ and spectral clustering on $\bu_{i}$, which is the infeasible and ideal data after factor adjustments,  in \prettyref{fig:SC no factor} with $d=5,20,100,500$ respectively. We can see that when $d=5$, FASC does not perform as well as  the spectral clustering using the ideal data but still significantly outperform the case without factor adjustment as shown in \prettyref{fig:FASC VS Spectral clustering}. On the other hand, when $d=20,100,500$, the performance of FASC are comparable with the spectral method with ideal data. These phenomenons support our theory and show the effectiveness of \prettyref{algo:fasc} under relatively high-dimension scenarios.
\end{remark}
\begin{figure}
    \centering
    \includegraphics[width=\linewidth]{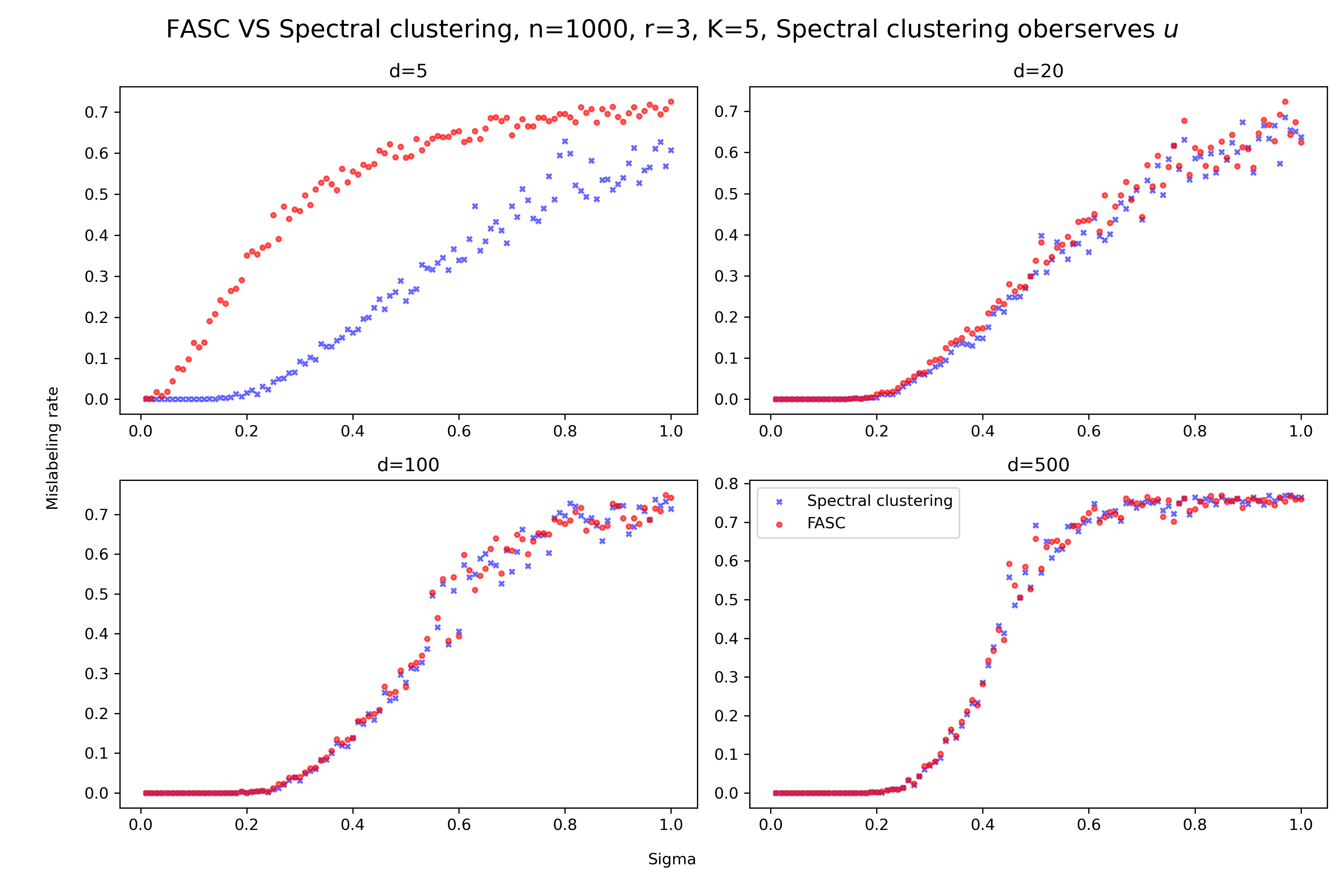}
    \caption{Mislabeling rates for FASC on $\boldsymbol{x}_i$ versus spectral clustering on $\boldsymbol{u}_i$, which is infeasible and ideal data after the factor adjustments.}
    \label{fig:SC no factor}
\end{figure}


\paragraph{Different number of factors in algorithm} 
In real datasets, the exact number $r$ of factors may be unknown. The following simulation study suggest that when applying \prettyref{algo:fasc}, choosing an $r$ slightly larger than the real number of factors can be helpful. In \prettyref{fig:FASC_with_different_r_model_d100_n1000_K5_r3}, the actual number of underlying factors is 3, and we apply FASC with $r=1,2,3,4,5$ respectively and plot the mislabeling rate as a function of $\sigma$. When the chosen number of factors in the algorithm is smaller than the actual number of factors (i.e., $r=1,2$), FASC performs poorly since the factor-adjusted data is still strongly dependent. On the other hand, when the chosen number of factors in the algorithm is larger than the actual number of factors (i.e., $r=4,5$), FASC provides much better performance.   It performs worse than the case with ideal $r=3$ due to the noise accumulation as elucidated in \cite{fan2008high}.
\begin{figure}
    \centering
    \includegraphics[width=1\textwidth, height=8cm, keepaspectratio]{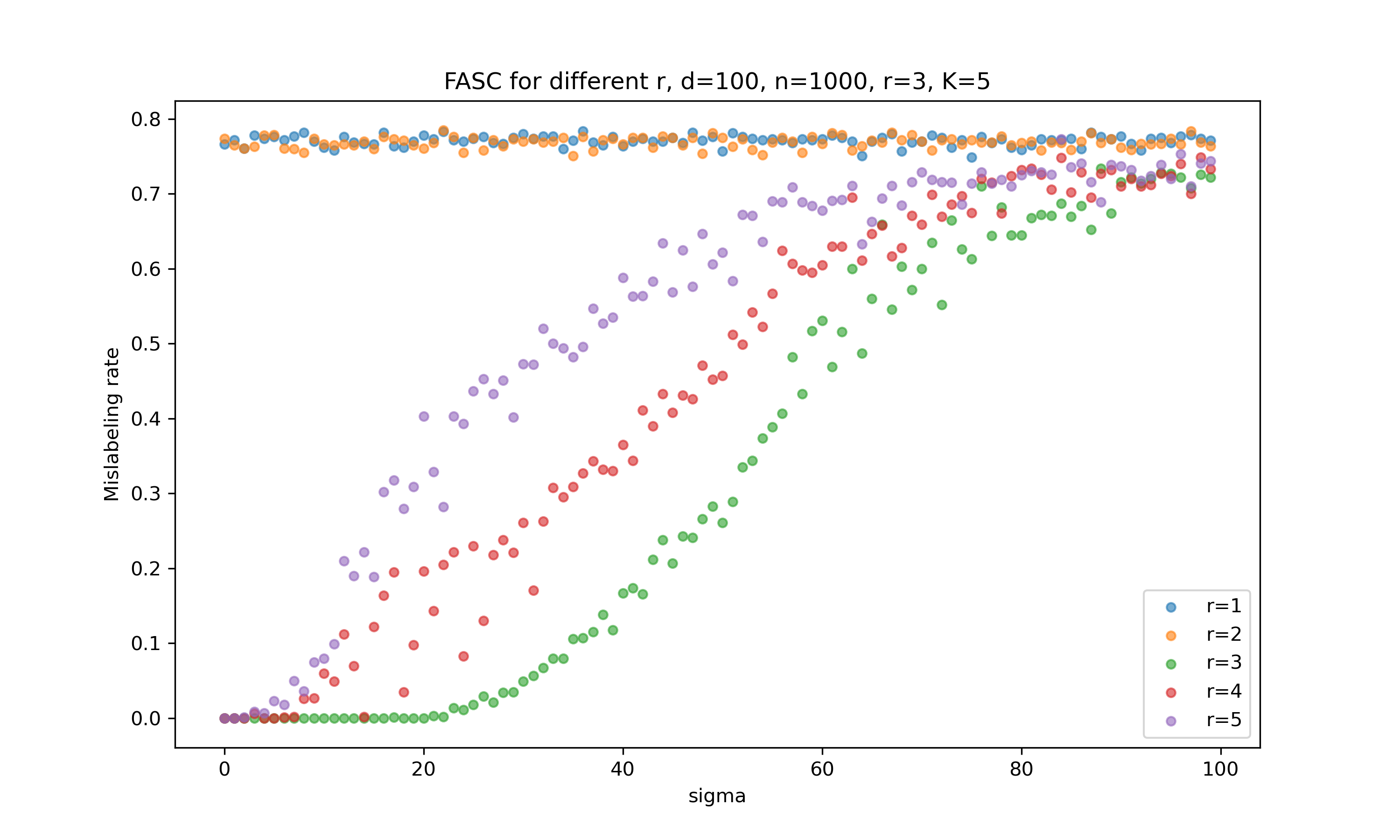}
    \caption{Mislabeling rates of FASC with a different number of factors.}
    \label{fig:FASC_with_different_r_model_d100_n1000_K5_r3}
\end{figure}

\subsubsection{High dimensional cases}
Some experiments are also provided for the scenario that $d$ is large compared to $n$ with the same data-generating process. We choose $n=100$, and $d=500,2000$ respectively. The results are shown in \prettyref{fig:FASC VS Spectral clustering, high dimension}, where FASC still performs much better than the vanilla spectral clustering. However, when $d=2000$, FASC can only achieve exact recovery when $\sigma \leq 0.1$, while in the previous section, e.g. \prettyref{fig:FASC VS Spectral clustering}, when $n=1000,d=100$, FASC can achieve exact recovery for some $\sigma \geq 0.2$. This phenomenon suggests that when $d$ is much larger than the sample size, the performance for FASC deteriorates. However, this deterioration is due to the nature of spectral clustering, which we will show by comparing the results with those of spectral clustering with infeasible ideal data. Similar to what we do in \prettyref{fig:SC no factor}, in the following experiments, we compare the results of FASC on $\bx_{i}$ and spectral clustering on $\bu_{i}$ in \prettyref{fig:FASC VS Spectral clustering on u, high dimension} with $n=100$ and $d=500,2000$ respectively. We can see that the performance of FASC are still comparable with the spectral method with ideal data, even under the scenario that $d$ is large compared to $n$, showing the effectiveness of FASC. 

\begin{figure}
    \centering
    \includegraphics[width=\linewidth]{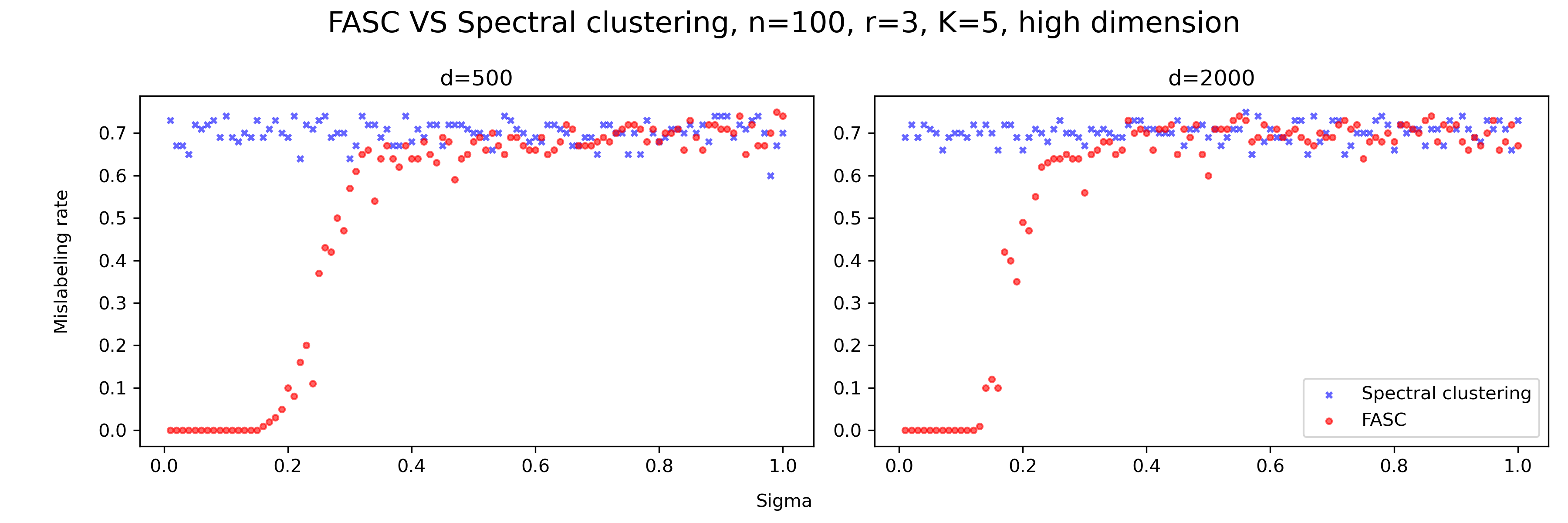}
    \caption{The mislabeling rates for FASC and vanilla spectral clustering in high dimension.}
    \label{fig:FASC VS Spectral clustering, high dimension}
\end{figure}

\begin{figure}
    \centering
    \includegraphics[width=\linewidth]{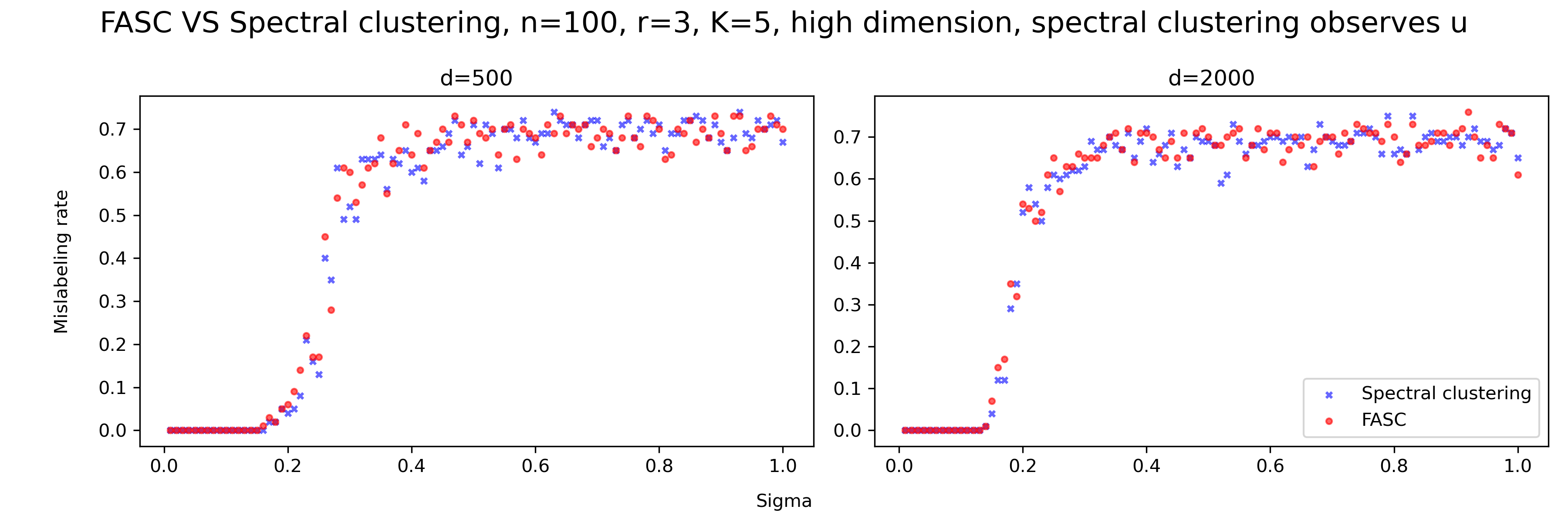}
    \caption{FASC on $\bx_i$ versus spectral clustering on $\bu_i$, in high dimension.}
    \label{fig:FASC VS Spectral clustering on u, high dimension}
\end{figure}

\subsection{Weak factor cases}
In this section, we will investigate the ``weak factor" cases. The data generating process remains the same, except that we generate the factor loading matrix $\bB\in \mathbb{R}^{d \times r}$ whose rows are drawn from \iid $\frac{1}{\sqrt{d}}\mathcal{N}(0,\bI_{r})$.
The simulation results are provided in \prettyref{fig:FASC VS Spectral clustering, weak factor} for $n=1000$, $d=100,500$. Similar to \prettyref{fig:FASC VS Spectral clustering}, FASC still performs well under the weak factor scenario, supporting our result in \prettyref{cor:weak}.

\begin{figure}
    \centering
    \includegraphics[width=\linewidth]{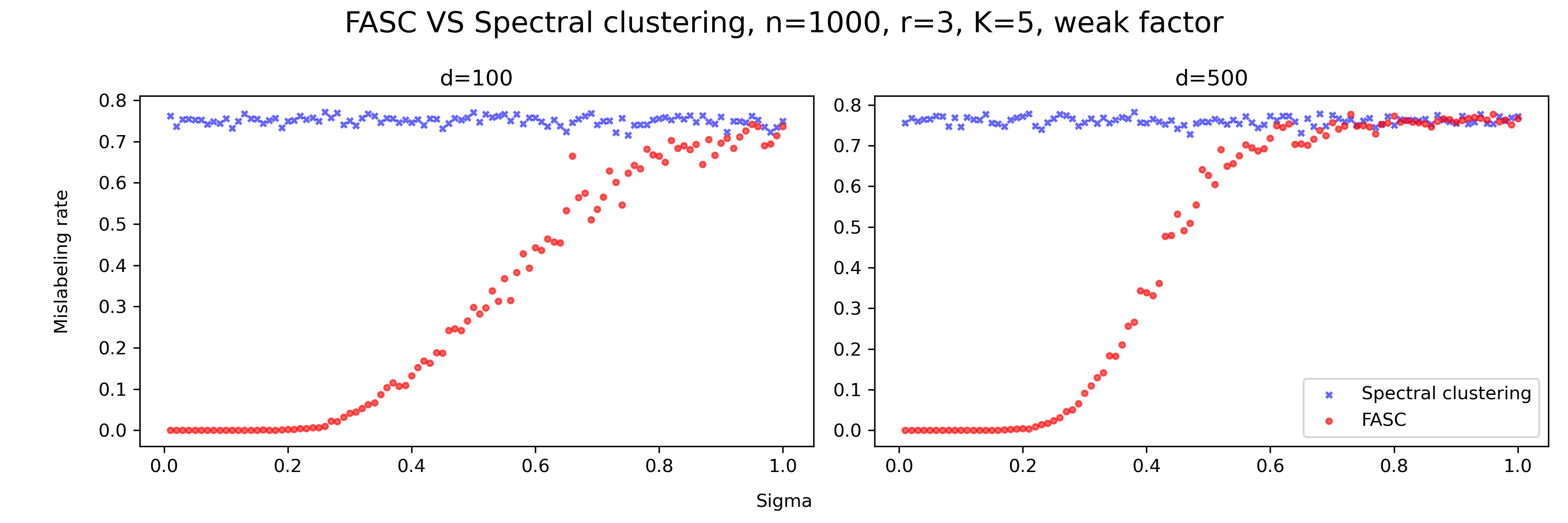}
    \caption{FASC versus vanilla spectral clustering for the case with weak factors.}
    \label{fig:FASC VS Spectral clustering, weak factor}
\end{figure}

\section{Real data analysis}
\label{sec:real_data}
\subsection{Mice protein expression}
We analyze the mice protein expression data from UCI Machine Learning Repository:\\ \url{https://archive.ics.uci.edu/ml/datasets/Mice+Protein+Expression}.\\
The dataset consists of the expression levels of 77 proteins measured in the cerebral cortex of 8 classes of control (which we view as the true class label), containing 1080 measurements per protein. Each measurement can be considered as an independent sample. We first cleaned the missing values by dropping 6 columns ``BAD\_N", ``BCL2\_N", ``pCFOS\_N", ``H3AcK18\_N", ``EGR1\_N", ``H3MeK4\_N" and then removed all rows with at least 1 missing value. After data cleaning and centralization, we get a dataset of 1047 samples each having 71 attributes. We then apply the $K$-means clustering (we use the standard kmeans function in R which runs the Hartigan-Wong algorithm \citep{hartigan1979algorithm}), spectral clustering and FASC to the unlabeled dataset with the number of clusters $K=8$. We compare the clustering results to the true class labels and calculate the mislabeling rate. We also compare them with the ``random guess" baseline, i.e. the case where we assign each data point randomly to one of the eight clusters and then compute the mislabeling error. The clustering results are stated in the second column of Table \ref{table1}, where FASC with $r=1$ or $2$ performs much better than $K$-means and spectral clustering. However, for larger $r$, the performance of FASC is not better than $K$-means or spectral clustering.
\begin{table}
  \centering
  \resizebox{\textwidth}{!}{%
  \begin{tabular}{p{5cm}|p{5cm}|p{5cm}}
    \hline
    \textbf{Clustering algorithm} & \textbf{Mislabeling rate for mice protein expression data} & \textbf{Mislabeling rate for DNA codon usage frequencies data} \\
    \hline
    $K$-means & 0.659 & 0.610  \\
    Spectral clustering & 0.657 & 0.608 \\
    \hline
    FASC($r=1$) & 0.538 & 0.507  \\
    FASC($r=2$) & 0.569 & 0.565 \\
    \hline
    FASC($r=3$) & 0.666 & 0.630 \\
    FASC($r=4$) & 0.645 & 0.700 \\
    \hline
    Baseline(random guess) & 0.875 & 0.875 \\
    \hline
  \end{tabular}%
  }
\caption{Misclassification rate for different methods on mice protein expression data and DNA codon usage frequencies data}
  \label{table1}
\end{table}
This phenomenon can be explained by the spectral structure of this dataset. We draw the scree plot \prettyref{fig:scree_plot_1} which shows the eigenvalues of the dataset arranged in descending order. It is obvious that the first two eigenvalues are significantly larger than others, suggesting that there may exist two significant factors.  Therefore it is appropriate to do factor adjustment on this dataset with the number of factors $r=1$ or $2$, justifying the result that FASC performs well with $r=1$ or $2$.

\begin{figure}
\makebox[\textwidth][c]{
\begin{subfigure}[b]{0.5\textwidth}
  \centering
  \includegraphics[width=1\textwidth, height=8cm, keepaspectratio]{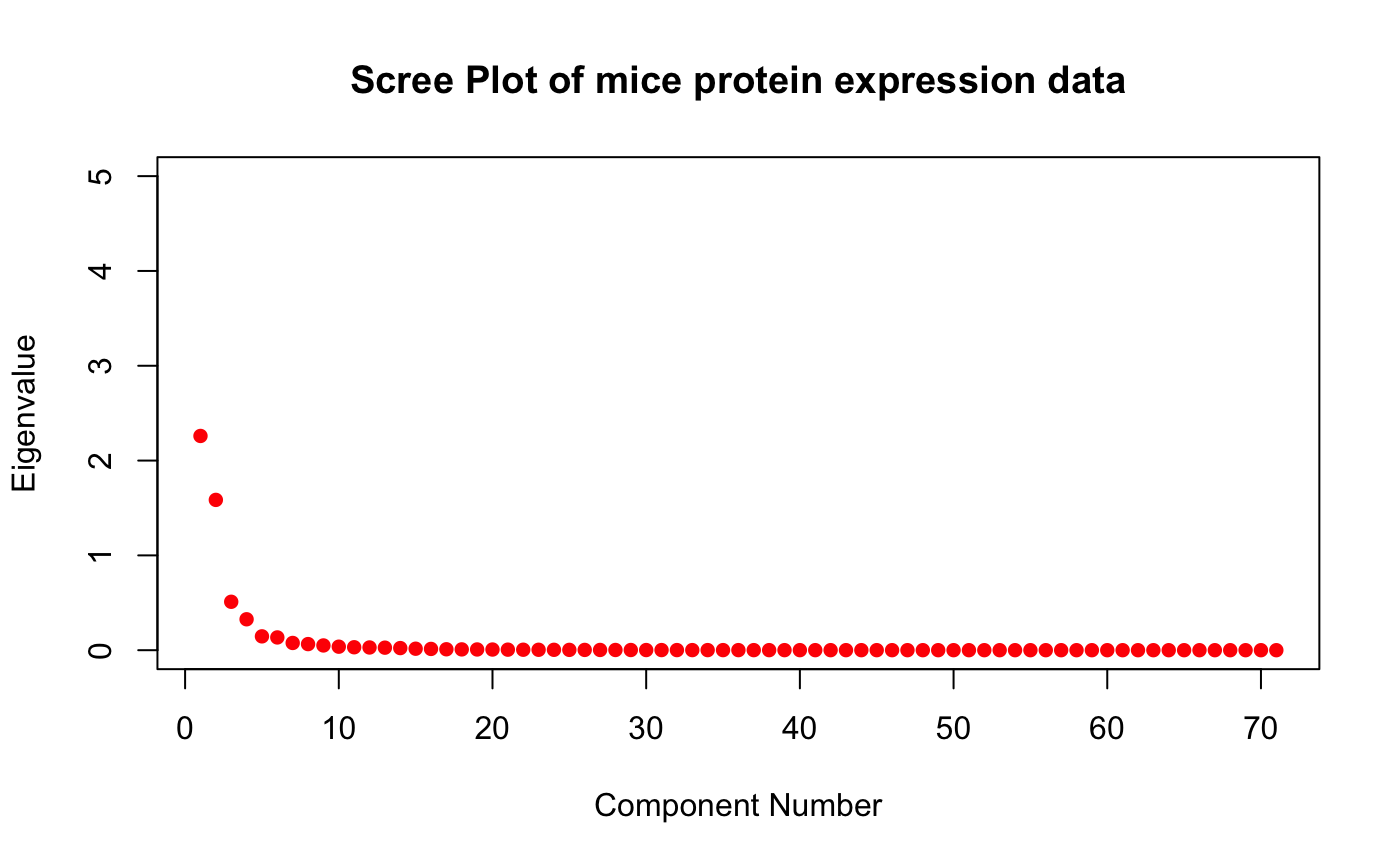}
  \caption{Scree plot of mice protein expression data}
  \label{fig:scree_plot_1}
\end{subfigure}
\hfill
\begin{subfigure}[b]{0.5\textwidth}
  \centering
  \includegraphics[width=1\textwidth, height=8cm, keepaspectratio]{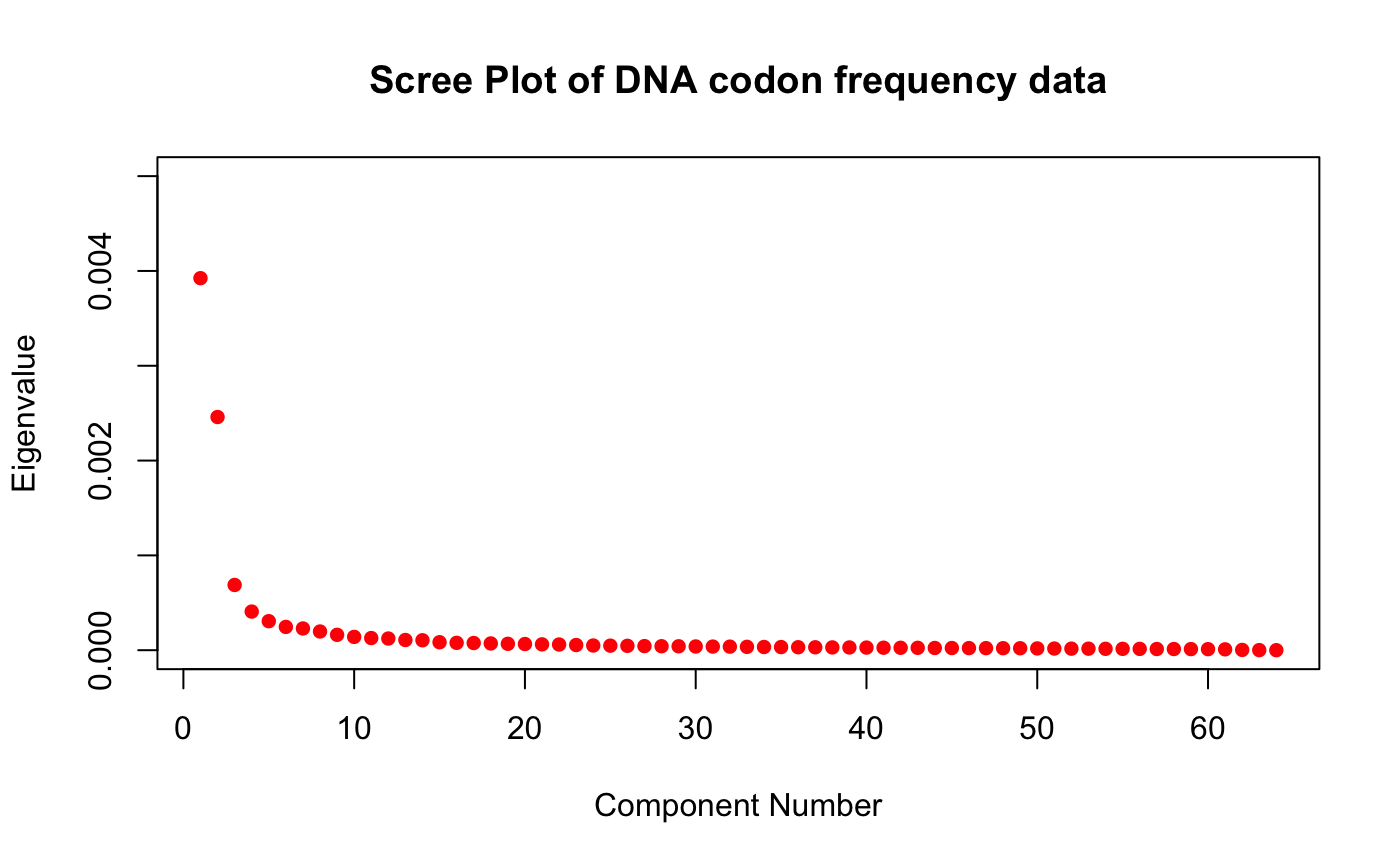}
  \caption{Scree plot of DNA codon frequency data}
  \label{fig:scree_plot_2}
\end{subfigure}
}
\caption{Scree plot of real data}
\label{scree_plot}
\end{figure}
\subsection{DNA codon usage frequencies}
We analyze the DNA codon usage frequencies data from UCI Machine Learning Repository:\\ \url{https://archive.ics.uci.edu/dataset/577/codon+usage}.\\
The dataset consists of codon usage frequencies in the genomic coding DNA of a large sample of diverse organisms from different taxa tabulated in the CUTG database. Each sample has 64 attributes that are the codon (including `UUU', `UUA', `UUG', `CUU', etc) frequency. The data are classified by their kingdom (which we view as the true class label): `bct'(bacteria), `pln' (plant), `inv' (invertebrate), `vrt' (vertebrate), `mam' (mammal), `rod' (rodent), `pri' (primate) and `vrl'(virus) (`arc',`phg',`plm' are removed since the number of samples is very small). After data cleaning, we get a dataset of 12135 samples with 64 attributes. We apply the $K$-means clustering, spectral clustering and FASC to the centralized unlabeled dataset with the number of clusters $K=8$. We compare the clustering results to the true class label and calculate the misclassification rate. The results of clustering are stated in the third column of Table \ref{table1}, where FASC with $r=1$ or $2$ performs much better than $K$-means and spectral clustering.
Similar to the mice protein expression analysis, this phenomenon can also be explained by the spectral structure. In the scree plot \prettyref{fig:scree_plot_2}, the first two eigenvalues are significantly larger than others. 

\section{Proof sketch of \prettyref{thm:general}}
\label{sec:proof_sketch}

Below we provide a brief argument to explain why the FASC algorithm can guarantee a consistent clustering. On a high level, we can summarize \prettyref{asmp:factors_genr} in the following way:
\begin{itemize}
	\item the covariance component corresponding to the factors is the most significant term for explaining the data variability $\cov(\bx_i)$.
	
	\item the factor components are approximately orthogonal to the idiosyncratic components,
\end{itemize}
Our proof shows that these assumptions are sufficient to guarantee that the idiosyncratic components can be extracted from the data. In view of \eqref{eq:ui-representation} we can conclude that the idiosyncratic components contains all the information regarding cluster memberships, so a consistent estimation of the idiosyncratic components should suffice to extract a consistent clustering. 
\paragraph{First phase (extracting the idiosyncratic components):}

To estimate $\bu_i$ we remove the component of $\bx_i$ that lies on the column space of $\bB$. We first estimate the matrix of projection on the column space of $\bB$. The population covariance matrix $\bSigma$ under model \eqref{eq:factor_cluster} is given by
\begin{align}\label{eq:fasc1}
	\bSigma:=\text{cov}(\bx_{i}) = \text{cov} (\bB\bff_{i}+\bu_{i})=\bB\bB^\sfT +  \bSigma_u,\quad
	\bSigma_u = \text{cov}(\bu_{i}).
\end{align} 
To extract the factor component, in view of the above decomposition we use the first $r$ principle components of the sample covariance matrix 
$$
\hat{\bSigma}:=\frac{1}{n}\bX^\sfT\bX = \frac{1}{n}\sum_{i=1}^{n} \bx_{i} \bx_{i}^\sfT.
$$ 
Let $\hat{\bV}_{r}$ denote the top $r$ eigenvectors of the eigen-decomposition of $\hat\bSigma$
$$\hat{\bV}_{r}:= (\hat{\bv}_{1},\cdots,\hat{\bv}_{r}),
\quad \hat{\bSigma}=\sum_{j=1}^{d}\hat{\lambda}_{j}\hat{\bv}_{j}\hat{\bv}_{j}^\sfT,
\quad \hat{\lambda}_{1} \geq \hat{\lambda}_{2} \geq \cdots \geq \hat{\lambda}_{d} \geq 0.$$
\prettyref{algo:fasc} uses the estimator of $\bu_{i}$ given by $\hat{\bu}_{i} = (\bI - \hat{\bV}_{r}\hat{\bV}_{r}^\sfT)\bx_{i}$. In the case where all the non-zero eigenvalues of $\bB\bB^\sfT$ are significantly larger than the operator norm of $\bSigma_\bu$ and $\bB\in \reals^{d\times r}$ has a full column rank $r$, the matrix $\hat \bV_r\hat \bV_r^\sfT$ should approximate the projection matrix onto the column space of $\bB$. Therefore $\hat{\bV}_{r}\hat{\bV}_{r}^\sfT\bx_{i}$ is approximately equal to $\bB\bff_{i}$, and hence, $\hat{\bu}_{i}$ is approximately equal to $\bu_i$. 

\paragraph{Second phase (spectral clustering on the estimated idiosyncratic components):}

To show that an application of \prettyref{algo:spectral} on $\{\hat{\bu}_{i}\}_{i=1}^{n}$ will give us a consistent clustering, we will use the following guarantee for spectral clustering with sub-Gaussian data.
\begin{lemma}{\cite[Theorem 3.1]{zhang2022leaveoneout}}
	\label{lmm:spectral_clustering_lemma}
	Consider the model 
	\begin{align}
		\bz_{i} = \btheta_{y_{i}} + \bepsilon_{i},\quad y_i\in [K], i\in [n]. 
	\end{align}
	Let $\hat{\by} := (\hat{y}_{1}, \cdots, \hat{y}_{n})$ be the output of \prettyref{algo:spectral} on $\{\bz_i\}_{i=1}^n$. Assume $\bepsilon_{i} \sim \subG_{d}(\sigma^{2})$ independently with zero mean for each $i \in [n]$. Let $\beta= \frac{K}{n} \min_{a \in [K]} |\{i : y_i = a\}|$ with $\beta n/K^2 >10$. There exist constants $\tilde C, \tilde C' > 0$ such that under the assumption that
	\begin{equation*}
		\psi:=\frac{\min_{i\neq j\in [K]}\|\btheta_i-\btheta_j\|}{\beta^{-0.5}K(1+\sqrt{\frac{d}{n}})\sigma}>\tilde C,\quad
		\rho:=\frac{\sigma_k([\btheta_{y_1},\dots,\btheta_{y_n}]^\sfT)}{(\sqrt{n}+\sqrt{d})\sigma}>\tilde C,
	\end{equation*}
	we have 
	\begin{equation*}
		\EE [\mathcal{M}(\hat{\by}, \by)] \leq \exp (-(1-\tilde C'(\psi^{-1}+\rho^{-2}))\frac{\min_{i\neq j\in [K]}\|\btheta_i-\btheta_j\|^2}{8\sigma^2}) + \exp(-\frac{n}{2}).
	\end{equation*}
\end{lemma}
The related proof details are provided in \prettyref{app:proof_details}.

\section{Conclusion}

We study the problem of clustering high-dimensional data generated from a mixture of strongly correlated distributions. We propose FASC, which decorrelates the dependence among variables and reduces the noise through factor adjustments. Theoretical analysis of FASC is provided, showing FASC has a misclassification rate of $e^{-\Omega(\snr)}$ with near linear sample complexity $n = \tilde{\Omega}(d)$. We also show the applicability of FASC with real data experiments and numerical studies. The experimental results suggest that FASC achieves performance improvement compared with traditional methods, such as spectral clustering and $K$-means.

\bibliographystyle{agsm}
\bibliography{texfiles/fasc_references}

\newpage
\bigskip
\begin{center}
{\large\bf SUPPLEMENTARY MATERIAL to Factor Adjusted Spectral Clustering for Mixture Models}
\end{center}

\appendix

\section{Proof details of \prettyref{thm:general}}
\label{app:proof_details}

The proof of \prettyref{thm:general} will rely on the following major results. 
\begin{lemma} \label{lmm:bound_on_Delta_norm}
	Let $n \geq d \log^3 n$. Denote $\bDelta:=\bV_{r}\bV_{r}^{\sfT} -\hat{\bV}_{r}\hat{\bV}_{r}^{\sfT}$. Then there are constants $\xi_1>0,n_0\in \naturals^+$ such that for all $n\geq n_0$
	\begin{equation*}
		\PP\qth{\| \bDelta \|
		\leq  \frac{\xi_1}{\sigma_{\max}(\bB)} \pth{\sigma \vee \frac{1}{\sqrt{\log n}}}}\geq 1-n^{-10}.
	\end{equation*}
\end{lemma}

\begin{lemma} \label{lmm:results_on_mean_matrices}
	There are constant $\xi_2,\xi_3>0$ and $\xi_4\in (0,1)$ such that the following holds:
	\begin{enumerate}[label=(\roman*)]
		\item $\|(\bI-\bV_{r}\bV_{r}^{\sfT})\bB\| \leq \xi_2\pth{\sigma\vee \frac{1}{\sqrt{\log n}}}$
		\item $\lambda_{k} ((\bI-\bV_{r}\bV_{r}^{\sfT})\mathbb{E}[\bmu_{y_{i}}\bmu_{y_{i}}^{\sfT}](\bI-\bV_{r}\bV_{r}^{\sfT})) \geq \xi_3$
	 \item $\| (\bI-\bV_{r}\bV_{r}^{\sfT})(\bmu_{j_1} -\bmu_{j_2}) \| \geq \xi_4 \|\bmu_{j_1} -\bmu_{j_2}\|$ for any $j_{1}\neq j_{2}$.
	\end{enumerate}
\end{lemma}
Now we get back to proving \prettyref{thm:general}. Note that the final clustering is obtained when we apply \prettyref{algo:spectral} to $\{\hat\bu_i\}_{i=1}^n$. We will show that the estimates $\{\hat\bu_i\}_{i=1}^n$ will be distributed according to a mixture of approximately sub-Gaussian distribution and the centroids of the mixture components will be well separated. To show the above, we note the following decomposition
\begin{align}
	\hat\bu_i=\tilde\bmu_{y_i}+\tilde \bepsilon_i,\quad 
	\tilde{\bmu}_{j}:= (\bI-\hat{\bV}_{r}\hat{\bV}_{r}^{\sfT})\bmu_{j},\quad 
	\tilde{\bepsilon}_{i} := (\bI-\hat{\bV}_{r}\hat{\bV}_{r}^{\sfT}) \bB\bff_{i} + (\bI-\hat{\bV}_{r}\hat{\bV}_{r}^{\sfT}) \bepsilon_{i}.
\end{align}
Next, we establish the following guarantees with a probability at least $1-n^{-10}$
\begin{enumerate}[label=(P\arabic*)]
	\item \label{claim:P1} 
	$\min_{i\neq j\in [K]}\|\tilde\bmu_i-\tilde\bmu_j\|\geq \frac {\xi_4}8\min_{i\neq j\in [K]} \|\bmu_i-\bmu_j\|$ 
	\item \label{claim:P2} $\tilde\bepsilon_i\in \subG_d\pth{\xi_5\pth{\sigma^2\vee\frac 1{\log n}}}$ for a constant $\xi_5>0$.
\end{enumerate}
We first prove the claim \ref{claim:P1}.  We have for $j_{1} \neq j_{2}$, 
\begin{align}\label{eq:m5}
	\|\tilde{\bmu}_{j_{1}} - \tilde{\bmu}_{j_2}\| &= \| (\bI-\hat{\bV}_{r}\hat{\bV}_{r}^{\sfT})(\bmu_{j_1} -\bmu_{j_2}) \| \nonumber\\
	& \geq \| (\bI-\bV_{r}\bV_{r}^{\sfT})(\bmu_{j_1} -\bmu_{j_2}) \| - \|\bDelta (\bmu_{j_1} -\bmu_{j_2})\| \nonumber\\
	&\stepa{\geq} \xi_4 \|\bmu_{j_1} -\bmu_{j_2}\| - \|\bDelta\| \| \bmu_{j_1} -\bmu_{j_2}\| \nonumber\\
	& \stepb{\geq} \xi_4 \|\bmu_{j_1} -\bmu_{j_2}\| - \frac{\xi_1}{\sigma_{\max}(\bB)} \pth{\sigma \vee \frac{1}{\sqrt{\log n}}} \| \bmu_{j_1} -\bmu_{j_2}\|
	\stepc{\geq} \frac{\xi_4}{2} \|\bmu_{j_1} -\bmu_{j_2}\|.
\end{align}
where (a) follows from \prettyref{lmm:results_on_mean_matrices}, (b) follows from \prettyref{lmm:bound_on_Delta_norm}, and (c) follows from \prettyref{asmp:factors_genr}.

Next we will prove \ref{claim:P2}. We note the following decomposition for $\tilde{\bepsilon}_{i}$
\begin{align*}
	\tilde{\bepsilon}_{i} &= (\bI-\hat{\bV}_{r}\hat{\bV}_{r}^{\sfT}) \bB\bff_{i} + (\bI-\hat{\bV}_{r}\hat{\bV}_{r}^{\sfT}) \bepsilon_{i} = (\bI-\bV_{r}\bV_{r}^{\sfT}) \bB\bff_{i} + \bDelta \bB \bff_{i}+ (\bI-\hat{\bV}_{r}\hat{\bV}_{r}^{\sfT}) \bepsilon_{i} .
\end{align*}
Hence, conditioned on $\hat \bV_r,\bB$, and using the independence of $\bff_i$ and $\bepsilon_i$ we have that 
\begin{align}\label{eq:m1}
	\tilde\bepsilon_i
	&\in \subG_d((\|(\bI-\bV_{r}\bV_{r}^{\sfT}) \bB + \bDelta \bB\|)^{2} + \sigma^{2})
	\subseteq \subG_d(2\|(\bI-\bV_{r}\bV_{r}^{\sfT}) \bB\|^{2} +  2\|\bDelta \bB\|^{2} + \sigma^{2}).
\end{align} 
For a large constant $\xi_5>0$ we note that the bound $\|(\bI-\bV_{r}\bV_{r}^{\sfT}) \bB\|^{2} +  \|\bDelta \bB\|^{2} + \sigma^{2}\leq \xi_5\pth{\sigma^2\vee\frac 1{\log n}}$ holds with a probability at least $1-2n^{-10}$ as
\begin{itemize}
\item $\bff_{i} \in \subG_{r}(c_1),\bepsilon_{i} \in \subG_{d}(\sigma^2)$
\item Using $\|\bI-\hat\bV_r\hat\bV_r^\sfT\|\leq 1$, \prettyref{lmm:bound_on_Delta_norm}, \prettyref{lmm:results_on_mean_matrices} we get with a probability $1-2n^{-10}$
$$\|(\bI-\bV_{r}\bV_{r}^{\sfT}) \bB\|^{2}\leq \xi_2^2\pth{\sigma^2\vee \frac 1{{\log n}}},\quad
\|\bDelta \bB\|^{2} \leq \|\bDelta\|^2\cdot \sigma_{\max}^2(\bB)\leq \xi_1^2\pth{\sigma^2\vee \frac 1{\log n}}.$$
\end{itemize}
In view of the above and \eqref{eq:m1} we get that 
\begin{align}\label{eq:approx_subG}
	\tilde \bepsilon_i\in \subG_d\pth{\xi_5\pth{\sigma^2\vee \frac 1{\log n}}} \text{with a probability at least } 1-2n^{-10}.
\end{align}

Now we are ready to apply \prettyref{lmm:spectral_clustering_lemma} with the vectors $\{\hat \bu_i\}_{i=1}^n$. From now on our argument will be conditional on the values of $\hat\bV_r$ over the event $\sth{\hat \bV_r: \|\bDelta\|\leq \frac{\xi_1}{\sigma_{\max}(\bB)} \pth{\sigma \vee \frac{1}{\sqrt{\log n}}}}$ for the constant $\xi_1>0$ as in \prettyref{lmm:bound_on_Delta_norm}, and note that the above event holds with a probability $1-n^{-10}$. We will show that given any such $\hat\bV_r$ the final result will be the same, and hence the result without any conditioning on $\hat\bV_r$ will also hold.

In view of \eqref{eq:approx_subG} it remains to check whether the conditions on $\beta,\psi,\rho$ required by \prettyref{lmm:spectral_clustering_lemma} are satisfied. In our setup of \eqref{eq:factor_cluster} note that using a Chernoff type argument \cite[Section 2.2]{boucheron2003concentration} it can be shown that $\abs{\min_{a\in [K]}\{a:y_i=j\}}>\tilde c n$ with a probability $1-K{n^{-10}}$ for a small constant $\tilde c>0$. The above implies 
\begin{align}\label{eq:m2}
	\tilde c< \beta = \frac{K}{n} \min_{a \in [K]} |\{i : y_i = a\}| \leq 1 \quad  \text{ for all large } n.
\end{align} 
Next we check whether the condition on $\psi$ is satisfied. In view of \ref{claim:P1}, \eqref{eq:m2}, \prettyref{asmp:regularity} and $K<\tilde C$ for some constant $\tilde C>0$ we get that there is a large constant $\tilde C_1>0$ such that for sufficiently large $n$ and small enough $\sigma$
\begin{align}\label{eq:m3}
	\psi:=\frac{\min_{i\neq j\in [K]}\|\tilde\bmu_i-\tilde\bmu_j\|}{\beta^{-0.5}K(1+\sqrt{\frac{d}{n}})\sigma}> \tilde C_1.
\end{align}
Next we check the condition on $\rho$ in \prettyref{lmm:spectral_clustering_lemma}. For observations $\{\hat \bu_i\}_{i=1}^n$, $\rho$ is defined as
\begin{equation*} \rho:=\frac{\sigma_k([\tilde\bmu_{1},\dots,\tilde\bmu_n]^\sfT)}{(\sqrt{n}+\sqrt{d})\sigma}.
\end{equation*}
Let $\hat{\bM}$ denote the matrix $[\tilde\bmu_{1},\dots,\tilde\bmu_n]^\sfT$. To this end we note the following result
\begin{lemma} \label{lmm:eigenvalue_concentration}
	Conditioned on $\hat \bV_r$, there are constants $\tilde C_2,\tilde C_3, \tilde C_4>0$ such that for all large enough $n$, with probability $1-\tilde C_2/n$,
	 $$\lambda_{k}(\frac1n\sum_{i=1}^{n}\tilde{\bmu}_{y_i} \tilde{\bmu}_{y_i}^{\sfT}) \geq \lambda_{k}(\EE[\tilde{\bmu}_{y_i}\tilde{\bmu}_{y_i}^{\sfT}]) - \tilde C_3\frac{\log n}{\sqrt{n}}\geq \tilde C_4.$$
\end{lemma}
A proof of the above result is provided later. In view of the above we have 
\begin{align*}
	\frac{1}{\sqrt{n}}\sigma_{k}(\hat{\bM}) 
	= \sqrt{\frac1n \lambda_{k}(\hat{\bM}^{\sfT}\hat{\bM})}
	= \sqrt{\lambda_{k}(\frac1n\sum_{i=1}^{n}\tilde{\bmu}_{y_i} \tilde{\bmu}_{y_i}^{\sfT})}
	 \geq \sqrt{\lambda_{k}(\EE\qth{\tilde{\bmu}_{y_i}\tilde{\bmu}_{y_i}^{\sfT}}) -\tilde C_3\frac{\log n}{\sqrt{n}} }
	 > \tilde C_5.
\end{align*}
for some constant $\tilde C_5>0$ when $\sigma$ is small enough, with probability $1-\tilde C_2/n$. In view of the last display we get for a constant $\tilde C_6>0$, whenever $\sigma$ is small
\begin{align}\label{eq:m4}
	\rho=\frac{\sigma_k([\tilde\bmu_{1},\dots,\tilde\bmu_n]^\sfT)}{(\sqrt{n}+\sqrt{d})\sigma}
	\geq \frac {\tilde C_5}\sigma>\tilde C_6.
\end{align}
In view of \eqref{eq:m3}, \eqref{eq:m4}, and the claims \ref{claim:P1},\ref{claim:P2} we apply \prettyref{lmm:spectral_clustering_lemma} to get
\begin{equation*}
	\EE \qth{\mathcal{M}(\hat{\by}, \by)} \leq \exp \pth{-(1-C'(\psi^{-1}+\rho^{-2}))\frac{\min_{i\neq j\in [K]}\|\tilde\bmu_i-\tilde\bmu_j\|^2}{8\pth{\sigma^2\vee\frac 1{\log n}}}} + \exp(-\frac{n}{2}).
\end{equation*}
In view of \eqref{eq:m5}, \eqref{eq:m3} and \eqref{eq:m4} we get that there exists a constant $\tilde C_7>0$ for which
\begin{equation*}
	\EE \qth{\mathcal{M}(\hat{\by}, \by)} \leq \exp \pth{-{\tilde C_7\cdot  {\overline{\snr}^2}\over 1\vee \frac 1{\sigma\log n}}} + \exp(-\frac{n}{2}),\quad 
	\overline{\snr} := \frac{\min_{i\neq j}\|\bmu_{i}-\bmu_{j}\|^{2}}{\sigma^2}.
\end{equation*}
Notice that all the above results holds are conditioned on a fixed $\bDelta$ that satisfies \prettyref{lmm:bound_on_Delta_norm}. For each of those $\bDelta$, the constants remain the same. Therefore the above results hold simlutaneously for any $\bDelta$ that satisfies \prettyref{lmm:bound_on_Delta_norm}. Therefore with probability greater than $1-n^{-10}$, for all sufficiently large $n$, the above results will hold without conditioning on $\bDelta$. This concludes the proof of \prettyref{thm:general} as $\sigma<1$.

%
%

\begin{remark}
Although \cite[Theorem 3.1 and Corollary 3.1]{zhang2022leaveoneout} hold for spectral clustering with exact $K$-means, the results will actually hold for $(1+\epsilon)$-approximate K-means with only differences in constants. To prove this we use an argument similar to the proof of \cite[Theorem 2.2]{loffler2021optimality}. Notice that the proof of \cite[Proposition 3.1]{zhang2022leaveoneout} claims that 
\begin{equation*}
    \|\hat{\Theta}-\hat{P}\|_{F} \leq \|P-\hat{P}\|_{F}
\end{equation*}
since $\hat{\Theta}$ is the solution to the exact $K$-means. However if $\hat{\Theta}$ is the solution to $(1+\epsilon)$-approximate $K$-means, we can similarly get 
\begin{equation*}
    \|\hat{\Theta}-\hat{P}\|_{F} \leq (1+\epsilon) \|P-\hat{P}\|_{F}.
\end{equation*}
Using this inequality will not affect the analysis except for constant factors in the result statement in \prettyref{thm:general}.  
\end{remark}

\section{Technical results}

\label{app:technical_details}

\begin{lemma}\label{lmm:matrix_perturbation_lemma}
Consider the model defined in (\ref{eq:factor_cluster}). Let $\bSigma := \EE\qth{\bx_{i}\bx_{i}^{\sfT}}$ be the population covariance matrix, and $\hat{\bSigma}:= \frac{1}{n}\sum_{i=1}^{n} \bx_{i} \bx_{i}^{\sfT}$ be the empirical covariance matrix and let $n \geq d$. Then there are constants $n_0,\bar C_1,\bar C_2>0$ such that for all $n\geq n_0$ we have
\begin{equation*}
    \|\hat{\bSigma}- \bSigma\| \leq \bar C_1\pth{\lambda_{1}^{*} \sqrt{\frac{\log n}{n}} + \sqrt{\frac{\lambda_{1}^{*}d \log n}{n}} + \sqrt{\frac{d \log n}{n}} + \frac{d \log^2 n}{n}}
\end{equation*}
with probability greater than $1-\bar {C}_2 n^{-9}$. Here $\lambda_{1}^{*} = \lambda_{\max}(\bB \bB^{\sfT})=\sigma_{\max}(\bB)^2$.
\end{lemma}

\begin{proof}
	We will need the following result, which is a truncated version of matrix Bernstein inequality:
	\begin{lemma}{\cite[Corollary 3.2]{MAL-079}}\label{lmm:matrix_Bernstein}
		Let $\bZ, \bZ_1, \cdots, \bZ_n$ be i.i.d. random matrices with dimension $m_1\times m_2$. Suppose the followings hold: 
		\begin{equation}\label{Bernstein_condition}
			\PP(\|\bZ - \EE[\bZ]\| \geq L) \leq q_0 \quad \mbox{and} \quad  \| \EE[\bZ]-\EE[ \bZ \11 \{ \|\bZ\| <L  \}   ]   \|  \leq q_1 .
		\end{equation}
		Define
  $$V= \max \{ \|\EE \qth{(\bZ - \EE \bZ )(\bZ - \EE \bZ )^{\sfT}}\|^{1/2}, \|\EE\qth{(\bZ - \EE \bZ )^{\sfT} (\bZ - \EE \bZ )}\|^{1/2} \}.$$ Set $m:=\max\{m_1,m_2\}$. Then for any $a \geq 2$, with probability at least $1-2m^{-a+1}-nq_0$, 
		\begin{align*}
			\|\frac{1}{n}\sum_{i=1}^{n} ( \bZ_i - \EE \bZ_i )\| \leq V \sqrt{\frac{2a \log m}{n}} + L  \frac{2a \log m}{3 n}  + q_1,    
		\end{align*}
		where $m=m_1+m_2$.
	\end{lemma}
	Back to the proof of \prettyref{lmm:matrix_perturbation_lemma}, we have the following decomposition: 
	\begin{align} \label{matrix_perturbation_ineq1}
		\|\hat{\bSigma}- \bSigma\| &= \|\frac{1}{n}\sum_{i=1}^{n}\pth{\bx_i\bx_i^{\sfT} - \EE \qth{\bx_i\bx_i^{\sfT}}}\| \notag \\
		&= \|\frac{1}{n}\sum_{i=1}^{n}(\bB \bff_i + \bu_i)(\bB \bff_i + \bu_i)^{\sfT} - \EE \qth{(\bB \bff_i + \bu_i)(\bB \bff_i + \bu_i)^{\sfT}}\| \notag \\
		& \leq \|\bB\|^{2} \|\frac{1}{n}\sum_{i=1}^{n} \bff_i \bff_i^{\sfT} - \EE \qth{\bff_i \bff_i^{\sfT}}\| + 2\|\bB\| \|\frac{1}{n}\sum_{i=1}^{n}\bff_i \bu_i^{\sfT}\| + \|\frac{1}{n}\sum_{i=1}^{n} \bu_i \bu_i^{\sfT} - \EE \qth{ \bu_i \bu_i^{\sfT}}\|.
	\end{align}
	We will then bound each of the three terms by matrix Bernstein inequality. For $\|\frac{1}{n}\sum_{i=1}^{n} \bff_i \bff_i^{\sfT} - \EE \qth{\bff_i \bff_i^{\sfT}}\|$, we apply Lemma \ref{lmm:matrix_Bernstein} by setting $\bZ_i=\bff_i\bff_i^{\sfT}$. We first choose proper $L$, $q_0$ and $q_1$ for (\ref{Bernstein_condition}) to hold. Recall from \prettyref{asmp:model} (a) that $\bff_i \in \subG_r(c_1^2)$ and $\EE\qth{\bff_i}=0$ implies $\|\bff_i\| \in \subG (r {\tilde{c}}^2)$ for some constant ${\tilde{ c}}>0$ \cite[Lemma 1]{jin2019short}. Therefore 
	\begin{align*}
		\PP(\|\bZ - \EE[\bZ]\| \geq L) & = \PP(\|\bff_i \bff_i^{\sfT} - \EE[\bff_i \bff_i^{\sfT}]\| \geq L) \\
		&\leq \PP(\|\bff_i \bff_i^{\sfT}\| +\| \EE[\bff_i \bff_i^{\sfT}]\| \geq L) \\
		&= \PP(\|\bff_i \bff_i^{\sfT}\| \geq L-1) 
		\leq \PP(\|\bff_i\|^2 \geq L-1) 
		= \PP(\|\bff_i\| \geq \sqrt{L-1}).
	\end{align*}
	Let $L= 22\tilde c^2 r \log n + 2$. Then by the sub-Gaussian tail bound
	\begin{align}
		Y\in \subG(\sigma^2),\quad \PP\qth{|Y-\EE\qth{Y}|>y}\leq 2e^{-{y^2\over 2\sigma^2}},
	\end{align}
	we have $\PP(\|\bff_i\| \geq \sqrt{L-1}) \leq 2 n^{-11} :=q_0$. For determining $q_1$ as in \prettyref{lmm:matrix_Bernstein} we note
	\begin{align*}
		\| \EE[\bZ]-\EE[ \bZ \11 \{ \|\bZ\| <L  \}  ]  \|  &=  \|\EE[ \bZ \11 \{ \|\bZ\| \geq L \} ] \| \\
		& \leq \EE[ \|\bZ\| \11 \{ \|\bZ\| \geq L \} ]  \\
		&= \EE[ (\|\bZ\|-L) \11 \{ \|\bZ\|-L \geq 0 \} ] + L \PP (\|\bZ\| \geq L ) \\
		& =  \int_{0}^{\infty}  \PP(\|\bZ\|-L \geq x) d x + L \PP (\|\bZ\| \geq L ) \\
		& = \int_{L}^{\infty}  \PP(\|\bff_i\bff_i^{\sfT}\| \geq y) d y + L \PP (\|\bff_i\bff_i^{\sfT}\| \geq L ) \\
		& \leq \int_{L}^{\infty}  \PP(\|\bff_i\|^2 \geq y) d y + L \PP (\|\bff_i\|^2 \geq L ) \\
		& \leq \int_{L}^{\infty} 2e^{-\frac{y}{2r{\tilde c}^2}}d y + 2Le^{-\frac{L}{2r{\tilde{ c}}^2}} \\
		&= (4rc^2+2L) e^{-\frac{L}{2r{\tilde{ c}}^2}}
		 \leq 3L e^{-\frac{L}{2r\tilde c^2}} \leq n^{-10} :=q_1
	\end{align*}
	for a sufficiently large $n$. Using the definition of $V$ we have 
	\begin{align*}
		V &= \max \{ \|\EE \qth{(\bZ - \EE \bZ )(\bZ - \EE \bZ )^{\sfT}}\|^{1/2}, \|\EE\qth{(\bZ - \EE \bZ )^{\sfT} (\bZ - \EE \bZ )}\|^{1/2} \} \\
		&= \|\EE \qth{(\bff_i \bff_i^{\sfT} - \EE \qth{\bff_i \bff_i^{\sfT}})(\bff_i \bff_i^{\sfT} - \EE \qth{\bff_i \bff_i^{\sfT}})}\|^{1/2}\\
        &= \|\EE \qth{\bff_i \bff_i^{\sfT} \bff_i \bff_i^{\sfT}} - (\EE \qth{\bff_i \bff_i^{\sfT}})^2\|^{1/2} 
        \leq \|\EE \qth{\bff_i \bff_i^{\sfT} \bff_i \bff_i^{\sfT}}\|^{1/2}.
	\end{align*}
	The last inequality is due to the fact that $\EE \qth{ \bff_i \bff_i^{\sfT} \bff_i \bff_i^{\sfT}}-(\EE \qth{\bff_i \bff_i^{\sfT}})^2$ and $(\EE \qth{\bff_i \bff_i^{\sfT}})^2$ are both positive semidefinite. Next, we notice that for any $\bv \in \RR^r$ such that $\|\bv\|=1$, 
	\begin{align*}
		\bv^{\sfT} (\EE \qth{\bff_i \bff_i^{\sfT} \bff_i \bff_i^{\sfT}}) \bv &= \EE [(\bff_i^{\sfT} \bv)^2 \bff_i^{\sfT} \bff_i] \\
		& \stepa{\leq} \sqrt{\EE \qth{(\bff_i^{\sfT} \bv)^4} \EE \qth{(\bff_i^{\sfT} \bff_i)^2 }} 
		\stepb{\leq} \tilde C_1 \sqrt{\EE \qth{\|\bff_i\|^4}}
		\stepc{\leq} \tilde C_2 r , 
	\end{align*}
	for constants $\tilde C_1,\tilde C_2>0$, where (a) follows using the Cauchy-Schwarz inequality, (b) follows as $\bff_i^{\sfT} \bv \in \subG(c_1^2)$, and (c) follows using $\|\bff_i\| \in \subG(\tilde c^2r)$. This implies $V \leq \sqrt{\tilde C_2 r}$. Therefore by \prettyref{lmm:matrix_Bernstein}, using $a=\frac{11 \log n}{\log r} $ we get for a constant $\tilde  C_3$
	\begin{equation} \label{matrix_perturbation_ineq2}
		\|\frac{1}{n}\sum_{i=1}^{n} \bff_i \bff_i^{\sfT} - \EE \qth{\bff_i \bff_i^{\sfT}}\| \leq C (\sqrt{\frac{\log n}{n}} + \frac{\log^2 n}{n})\leq \tilde C_3 \sqrt{\frac{\log n}{n}}
	\end{equation}
	with probability at least $1-3n^{-10}$ for all sufficiently large $n$, as $r$ is at most a constant. 
	
	For $\|\frac{1}{n}\sum_{i=1}^{n} \bff_i \bu_i^{\sfT} \|$, we apply Lemma \ref{lmm:matrix_Bernstein} by setting $\bZ_i=\bff_i \bu_i^{\sfT}$. We first choose proper $L$, $q_0$ and $q_1$ for (\ref{Bernstein_condition}) to hold. We claim that $\bu_i \in \subG_{d} (\tilde C_4^2)$ for a constant $\tilde  C_4>0$. This is because, $\bu_i = \bmu_{y_i} + \bepsilon_{i}$, so for any unit vector $\bv \in \RR^{d}$, $\bv^{\sfT} \bu_i = \bv^{\sfT}\bmu_{y_i} + \bv^{\sfT}\bepsilon_{i}$. Since $\bepsilon_i \in \subG_d (\sigma^2)$, we have $\bv^{\sfT}\bepsilon_{i} \in \subG (\sigma^2)$. On the other hand, notice that by \prettyref{asmp:regularity} (c) $c_5 \geq \max_{i\neq j}\|\boldsymbol{\mu}_{i}-\boldsymbol{\mu}_{j}\|$, and by \prettyref{asmp:model} (b), there is a constant $c\in (0,1)$ such that $\min_{j\in [K]}p_j>\frac c{K}$ and $\sum_{j=1}^{K}p_{j}\boldsymbol{\mu}_{j} = 0$. Therefore $\forall j \in [K]$,
 \begin{equation}
 \label{eq:bound_mu}
 \|\bmu_j\| = \| \bmu_j- \sum_{i=1}^{K} p_i \bmu_i\| = \| \sum_{i=1}^{K} p_i(\bmu_j - \bmu_i)\| \leq  \sum_{i=1}^{K} p_i \| \bmu_j - \bmu_i\| \leq \sum_{i=1}^{K} p_i c_5 = c_5.
 \end{equation}
Thus $| \bv^{\sfT}\bmu_{y_i}| \leq \|\bmu_{y_i}\| \leq c_5$. Then $\bv^{\sfT}\bmu_{y_i} \in \subG(\tilde C_4^2)$ for some constant $c$. Since $ \bv^{\sfT}\bmu_{y_i}$ and $\bv^{\sfT}\bepsilon_{i}$ are independent, we have $\bv^{\sfT} \bu_i = \bv^{\sfT}\bmu_{y_i} + \bv^{\sfT}\bepsilon_{i} \in \subG (\tilde  C_4^2 + \sigma^2)$, which implies
\begin{align}\label{eq:fasc2}
	\bu_i \in \subG_d (\tilde  C_5^2),
	\quad \|\bu_i\| \in \subG (d \tilde C_6^2)
\end{align}
for a constant $\tilde  C_6>0$ \cite[Lemma 1]{jin2019short}. Recall $\|\bff_i\| \in \subG (r c^2)$. Therefore $\|\bff_i \bu_i^{\sfT}\|$ is sub-Exponential \cite[Lemma 2.7.7]{vershynin2018high} with sub-Exponential norm $\tilde  C_7\sqrt{dr}$ for a constant $\tilde C_7>0$ \cite[Definition 2.7.5]{vershynin2018high}. In view of the properties of sub-Exponential distributions with parameter $\tilde  C_7\sqrt{dr}$ \cite[Proposition 2.7.1]{vershynin2018high} we get
	\begin{equation*}
		\PP(\|\bZ - \EE[\bZ]\| \geq L) = \PP(\|\bff_i \bu_i^{\sfT}\| \geq L) \leq 2n^{-11} :=q_0
	\end{equation*}
	For obtaining $q_1$ we use the following computation, 
	\begin{align*}
		\| \EE[\bZ]-\EE[ \bZ \11 \{ \|\bZ\| <L  \}  ]  \|  &=  \|\EE[ \bZ \11 \{ \|\bZ\| \geq L \} ] \| \\
		& \leq \EE[ \|\bZ\| \11 \{ \|\bZ\| \geq L \} ]  \\
		&= \EE[ (\|\bZ\|-L) \11 \{ \|\bZ\|-L \geq 0 \} ] + L \PP (\|\bZ\| \geq L ) \\
		& =  \int_{0}^{\infty}  \PP(\|\bZ\|-L \geq x) d x + L \PP (\|\bZ\| \geq L ) \\
		& = \int_{L}^{\infty}  \PP(\|\bff_i\bu_i^{\sfT}\| \geq y) d y + L \PP (\|\bff_i\bu_i^{\sfT}\| \geq L ) \\
		& \leq \int_{L}^{\infty} 2e^{-\frac{y}{2\tilde  C_7\sqrt{dr}}}d y + 2Le^{-\frac{L}{2\tilde  C_7\sqrt{dr}}} \\
		&= (4c^2\sqrt{dr}+2L) e^{-\frac{L}{2\tilde  C_7\sqrt{dr}}}
		\leq 3L e^{-\frac{L}{2\tilde  C_7\sqrt{dr}}}\leq n^{-10} :=q_1.
	\end{align*}
	where the last inequality holds for sufficiently large $n$. For the variance statistic $V$, by definition we have 
	\begin{align*}
		V &= \max \{ \|\EE\qth{ (\bZ - \EE \bZ )(\bZ - \EE \bZ )^{\sfT}}\|^{1/2}, \|\EE \qth{(\bZ - \EE \bZ )^{\sfT} (\bZ - \EE \bZ )}\|^{1/2} \} \\
		&= \max \{ \|\EE \qth{\bff_i \bu_i^{\sfT} \bu_i \bff_i^{\sfT}} \|^{1/2} , \|\EE \qth{\bu_i \bff_i^{\sfT} \bff_i \bu_i^{\sfT}} \|^{1/2} \} \\
		&\stepa{=} \max \{ \|\EE  \qth{\bu_i^{\sfT} \bu_i} \EE \qth{\bff_i \bff_i^{\sfT} }\|^{1/2} , \|\EE \qth{\bu_i \bu_i^{\sfT}} \EE \qth{\bff_i^{\sfT} \bff_i  }\|^{1/2} \} \\
		&= \max \{ \|\EE \qth{\| \bu_i\|^2} \bI_r \|^{1/2} , \|\EE \qth{bu_i \bu_i^{\sfT}} r  \|^{1/2} \} \\
		&\stepb{\leq} \tilde C_8\max\{\sqrt{d}, \sqrt{r}\} =\tilde C_8 \sqrt{d}, 
	\end{align*}
	for some constant $\tilde C_8>0$, where (a) follows as $\bu_i$ and $\bff_i$ are independent and (b) follows as $\|\bu_i\| \in \subG (d \tilde C_6^2)$ from \eqref{eq:fasc2} and 
	\begin{align}\label{eq:fasc3}
		\|\EE \qth{\bu_i \bu_i^{\sfT}}\|\leq \tilde C_9
	\end{align} 
	for a constant $\tilde C_9>0$. Hence, using \prettyref{lmm:matrix_Bernstein} with $a=\frac{11 \log n}{\log d} $ we get that for a constant $\tilde C_9>0$
	\begin{equation} \label{matrix_perturbation_ineq3}
		\|\frac{1}{n}\sum_{i=1}^{n} \bff_i \bu_i^{\sfT}\| \leq \tilde C_9 \sqrt{\frac{ d \log n}{n}}
	\end{equation}
	with probability at least $1-3n^{-10}$ for sufficiently large $n$.
	
	To bound $\|\frac{1}{n}\sum_{i=1}^{n} \bu_i \bu_i^{\sfT} - \EE \qth{\bu_i \bu_i^{\sfT}}\|$, we apply Lemma \ref{lmm:matrix_Bernstein} by setting $\bZ_i=\bu_i\bu_i^{\sfT}$. We first choose proper $L$, $q_0$ and $q_1$ for (\ref{Bernstein_condition}) to hold. Recall $\|\bu_i\| \in \subG (d \tilde C_6^2)$ from \eqref{eq:fasc2}. Therefore 
	\begin{align*}
		\PP(\|\bZ - \EE[\bZ]\| \geq L) & = \PP(\|\bu_i \bu_i^{\sfT} - \EE[\bu_i \bu_i^{\sfT}]\| \geq L) \\
		&\leq \PP(\|\bu_i \bu_i^{\sfT}\| +\| \EE[\bu_i \bu_i^{\sfT}]\| \geq L) \\
		& \stepa{\leq} \PP(\|\bu_i \bu_i^{\sfT}\| \geq L-\tilde C_9)
		\leq \PP(\|\bu_i\|^2 \geq L-\tilde C_9) 
		= \PP(\|\bu_i\| \geq \sqrt{L-\tilde C_9}),
	\end{align*}
	where (a) follows as $\| \EE[\bu_i \bu_i^{\sfT}]\| \leq \tilde C_9$ from \eqref{eq:fasc3}. Choose $L= 24\tilde C_6^2 d \log n + \tilde C_9$, by sub-Gaussian tail bound, we have $\PP(\|\bu_i\| \geq \sqrt{L-\tilde C_9}) \leq 2 n^{-12} :=q_0$. For $q_1$, 
	\begin{align*}
		\| \EE[\bZ]-\EE[ \bZ \11 \{ \|\bZ\| <L  \}  ]  \|  &=  \|\EE[ \bZ \11 \{ \|\bZ\| \geq L \} ] \| \\
		& \leq \EE[ \|\bZ\| \11 \{ \|\bZ\| \geq L \} ]  \\
		&= \EE[ (\|\bZ\|-L) \11 \{ \|\bZ\|-L \geq 0 \} ] + L \PP (\|\bZ\| \geq L ) \\
		& =  \int_{0}^{\infty}  \PP(\|\bZ\|-L \geq x) d x + L \PP (\|\bZ\| \geq L ) \\
		& = \int_{L}^{\infty}  \PP(\|\bu_i\bu_i^{\sfT}\| \geq y) d y + L \PP (\|\bu_i\bu_i^{\sfT}\| \geq L ) \\
		& \leq \int_{L}^{\infty}  \PP(\|\bu_i\|^2 \geq y) d y + L \PP (\|\bu_i\|^2 \geq L ) \\
		& \leq \int_{L}^{\infty} 2e^{-\frac{y}{2d\tilde C_6^2}}d y + 2Le^{-\frac{L}{2d\tilde C_6^2}} \\
		&= (4dc^2+2L) e^{-\frac{L}{2r\tilde C_6^2}}
		\leq 3L e^{-\frac{L}{2d\tilde C_6^2}}
		\leq n^{-10} :=q_1
	\end{align*}
	for sufficiently large $n$. Here we use the fact that $n \geq d$. For the variance statistic $V$, by definition we have 
	\begin{align*}
		V &= \max \{ \|\EE \qth{(\bZ - \EE \bZ )(\bZ - \EE \bZ )^{\sfT}}\|^{1/2}, \|\EE\qth{(\bZ - \EE \bZ )^{\sfT} (\bZ - \EE \bZ )}\|^{1/2} \} \\
		&= \|\EE \qth{\bu_i \bu_i^{\sfT} \bu_i \bu_i^{\sfT}} - (\EE \qth{\bu_i \bu_i^{\sfT}})^2\|^{1/2} \\
		& \leq \|\EE \qth{\bu_i \bu_i^{\sfT} \bu_i \bu_i^{\sfT}}\|^{1/2}.
	\end{align*}
	The last inequality is due to the fact that $\EE \qth{\bu_i \bu_i^{\sfT} \bu_i \bu_i^{\sfT}}-(\EE \qth{\bu_i \bu_i^{\sfT}})^2$ and $(\EE \qth{\bu_i \bu_i^{\sfT}})^2$ are both positive semidefinite. To bound $\|\EE \qth{\bu_i \bu_i^{\sfT} \bu_i \bu_i^{\sfT}}\|$, note that for any $\bv \in \RR^d,\|\bv\|=1$, 
	\begin{align*}
		\bv^{\sfT} (\EE \qth{\bu_i \bu_i^{\sfT} \bu_i \bu_i^{\sfT}}) \bv &= \EE [(\bu_i^{\sfT} \bv)^2 \bu_i^{\sfT} \bu_i] 
		\stepa{\leq} \sqrt{\EE \qth{(\bu_i^{\sfT} \bv)^4} \EE\qth{(\bu_i^{\sfT} \bu_i)^2} } 
		\stepb{\leq} \tilde C_{10} \sqrt{\EE \qth{\|\bu_i\|^4}}
		 \stepc{\leq} \tilde C_{11} d, 
	\end{align*}
	for constants $\tilde C_{10},\tilde C_{11}>0$, where (a) follows using the Cauchy-Schwarz inequality, (b) and (c) follow using \eqref{eq:fasc2}. This implies $V \leq \sqrt{\tilde C_{12}d}$. Therefore by Lemma \ref{lmm:matrix_Bernstein}, take $a=\frac{11 \log n}{\log d} $,
	\begin{equation} \label{matrix_perturbation_ineq4}
		\|\frac{1}{n}\sum_{i=1}^{n} \bu_i \bu_i^{\sfT} - \EE \qth{\bu_i \bu_i^{\sfT}}\| \leq \tilde C_{13} \pth{\sqrt{\frac{d\log n}{n}} + \frac{d\log^2 n}{n}}
	\end{equation}
	for a constant $\tilde C_{13}>0$ with probability at least $1-3n^{-10}$ for sufficiently large $n$. 
	
	Finally, combine (\ref{matrix_perturbation_ineq1}), (\ref{matrix_perturbation_ineq2}), (\ref{matrix_perturbation_ineq3}), (\ref{matrix_perturbation_ineq4}), we have
	\begin{equation*}
		\|\hat{\bSigma}- \bSigma\| \leq \bar C_1\pth{\lambda_{1}^{*} \sqrt{\frac{\log n}{n}} + \sqrt{\frac{\lambda_{1}^{*}d \log n}{n}} + \sqrt{\frac{d \log n}{n}} + \frac{d \log^2 n}{n}}
	\end{equation*}
	 with probability $1-\bar C_2n^{-10}$ for some constants $\bar C_1,\bar C_2>0$, as required.
\end{proof}

\subsection{Proof for \prettyref{lmm:justify_approx_perp_lemma}}
\label{app:justify_approx_perp_lemma_proof}
Suppose that $\bB$ is already given and fixed. Thus, $\bU$ is also fixed. Consider the Grassmannian manifold $G_{d,k}$ that consists of all $k$-dimensional subspaces of $\mathbb{R}^{d}$. Let $\bM \sim \text{Unif} (G_{d,k})$ (see \cite{vershynin2018high} for detailed discussion of the measure on Grassmannian manifold), which means $\bM$ represents a uniformly randomly chosen $k$-dimensional subspace of $\mathbb{R}^{d}$. Let the columns of $\bU$ be $(\bw_{1},\cdots,\bw_{r})$. Then using \cite[Lemma 5.3.2]{vershynin2018high} we can conclude
\begin{equation*}
	\mathbb{P} \left( \left|\|\bw_{i}^{T} \bM \|- \sqrt{\frac{k}{d}}\right| \geq t \right) \leq 2e^{-\xi d t^{2}}, 
	\quad 1\leq i\leq r,
\end{equation*}
for some constant $\xi>0$. Choosing $t=\sqrt{\frac{10 \log n}{\xi d}}$, the above implies that for each $i \in [r]$, with probability no less than $1-2\exp{(-\xi d t^{2})} = 1-2n^{-10}$,
\begin{equation*}
	\|\bw_{i}^{T} \bM\| \leq t + \sqrt{\frac{k}{d}} = \sqrt{\frac{10 \log n}{\xi d}} + \sqrt{\frac{k}{d}}.
\end{equation*}
Using the union bound, we have with a probability at least $1-2r n^{-10} \geq 1-n ^{-9}$ (for sufficiently large $n$), for every $i \in [r]$, we have $\|\bw_{i}^{T} \bM\| \leq \sqrt{\frac{10 \log n}{\xi d}} + \sqrt{\frac{k}{d}}$. Using 
$$\|\bU^{T}\bM\| \leq \sqrt{\sum_{i=1}^{r} \|\bw_{i}^{T} \bM\|^{2}} \leq \sqrt{r} \pth{\sqrt{\frac{10 \log n}{\xi d}} + \sqrt{\frac{k}{d}}}
\leq  \sqrt{\frac{11 r \log n}{\xi d}}$$ 
for sufficiently large $n$ (recall $r,k$ are smaller than some constant) we get the required results.

\subsection{Proof of \prettyref{lmm:bound_on_Delta_norm}}
\begin{proof}
We will use the following lemma (Weyl's inequality) for controlling the difference of eigenvalues between two matrices.
\begin{lemma}[Weyl's inequality for eigenvalues]\cite[Lemma 2.2]{MAL-079}
\label{lmm:weyl}
Let $\bA, \bE \in \mathbb{R}^{n \times n}$ be two real symmetric matrices. For every $1 \leq i \leq n$, the $i$-th largest eigenvalues of $\bA$ and $\bA + \bE$ obey
\[
|\lambda_i(\bA) - \lambda_i(\bA + \bE)| \leq \|\bE\|.
\]
\end{lemma}
We will also use the Davis-Kahan's Theorem \cite[Corollary 2.8]{MAL-079} for controlling the distance between two the subspaces spanned by eigenvectors. 
\begin{lemma}[Davis-Kahan's $\sin \theta$ theorem] \label{lmm:davis-kahan}
Let $\bM^*$ and $\bM = \bM^* + \bE$ be two $n \times n$ real positive semi-definite symmetric matrices. Let $\lambda_1 \geq \cdots \geq \lambda_{n} \geq 0$ be the eigenvalues of $\bM$,  $\lambda_1^* \geq \cdots \geq \lambda_{n}^*\geq 0$ be the eigenvalues of $\bM^*$, and $\bu_i$ (resp. $\bu_i^*$) stands for the eigenvector associated with the eigenvalue $\lambda_i$ (resp. $\lambda_i^*$). We denote $\bU := [\bu_1, \cdots, \bu_r] \in \RR^{n \times r}$, $\bU^* := [\bu_1^*, \cdots, \bu_r^*] \in \RR^{n \times r}$. If $\|\bE\| \leq (1 - 1/\sqrt{2})(\lambda_r^* - \lambda_{r+1}^*) $, then $\|\bU\bU^{\sfT} - \bU^*\bU^{*\sfT}\| \leq \frac{2\|\bE\|}{\lambda_r^* - \lambda_{r+1}^*}$.

\end{lemma}


	By \prettyref{lmm:davis-kahan}, since $\hat{\bV}_{r}$ consists of the top $r$ eigenvectors of $\hat{\bSigma}$ and $\bV_{r}$ consists of the top $r$ eigenvectors of $\bSigma$, we have
	\begin{equation*}
		\|\bDelta\| \leq \frac{2\|\hat{\bSigma}- \bSigma\|}{\lambda_{r}(\bSigma)-\lambda_{r+1}(\bSigma)},
	\end{equation*}
	where $\lambda_{r}(\bSigma)$ and $\lambda_{r+1}(\bSigma)$ are the $r$-th and $(r+1)$-th largest eigenvalues of $\bSigma$ respectively. Notice that $\bSigma=\bB\bB^{\sfT} + \bSigma_u$, therefore using \prettyref{lmm:weyl} we have
	\begin{equation*}
		\lambda_{r}(\bSigma)-\lambda_{r+1}(\bSigma) \geq \lambda_{r}(\bB\bB^\sfT)-\lambda_{r+1}(\bB\bB^\sfT)- 2\|\bSigma_u\|=\sigma_{\min}(\bB)^2 - 2\|\bSigma_u\|.  
	\end{equation*}
	Therefore from \prettyref{lmm:matrix_perturbation_lemma} and the \prettyref{asmp:factors_genr} gives us the result.
\end{proof}

\subsection{Proof of \prettyref{lmm:results_on_mean_matrices}}
\subsubsection{Non-isotropic noise}
\begin{proof} We first consider the case where the noise is not isotropic. Using \prettyref{asmp:model} and the decomposition $\mathbb{E}[\bmu_{y_{i}}\bmu_{y_{i}}^{\sfT}]=\bM\tilde \bLambda\bM^{\sfT}$ as in \eqref{eq:eigen-decomp}
we get 
\begin{align*}
    \bSigma := \mathbb{E} [{\bx_{i}\bx_{i}^{\sfT}}]
    & =  \mathbb{E} [(\bB\bff_{i}+ \bmu_{y_{i}} + \bepsilon_{i})(\bB\bff_{i}+ \bmu_{y_{i}} + \bepsilon_{i})^{\sfT}] \\
    & = \bB \bB^{\sfT} + \mathbb{E}[\bmu_{y_{i}}\bmu_{y_{i}}^{\sfT}] + \EE [\bepsilon_{i}\bepsilon_{i}^{\sfT}] \\
    &= \bU \bLambda \bU^{\sfT} + \bM\tilde \bLambda\bM^{\sfT} + \EE [\bepsilon_{i}\bepsilon_{i}^{\sfT}].
\end{align*}
Let $\tilde{\bSigma}$ denote the matrix $\bU \bLambda \bU^{\sfT} + (\bI-\bU\bU^{\sfT})\bM\tilde \bLambda\bM^{\sfT}(\bI-\bU\bU^{\sfT})^{\sfT}$. Notice that 
\begin{equation*}
\bU^{\sfT} (\bI-\bU\bU^{\sfT})\bM = (\bU^{\sfT}- \bU^{\sfT}\bU\bU^{\sfT}) \bM = (\bU^{\sfT}-\bU^{\sfT})\bM = 0.   
\end{equation*}
This implies that the column space of $ (\bI-\bU\bU^{\sfT})\bM\tilde \bLambda\bM^{\sfT}(\bI-\bU\bU^{\sfT})^{\sfT}$ (which belongs to the column space of $(\bI-\bU\bU^{\sfT})\bM$) is orthogonal to the column space of $\bU$. Thus we can write the spectral decomposition of $(\bI-\bU\bU^{\sfT})\bM \tilde \bLambda \bM^{\sfT}(\bI-\bU\bU^{\sfT})^{\sfT}$ as
\begin{align}\label{eq:m7}
     (\bI-\bU\bU^{\sfT})\bM\tilde \bLambda\bM^{\sfT}(\bI-\bU\bU^{\sfT})^{\sfT} = \bU_{\perp} \bar \bLambda \bU_{\perp}^{\sfT}
\end{align}
for a diagonal matrix $\bar{\bLambda}$ and an orthogonal matrix $\bU_{\perp}$ that satisfies $\bU_{\perp}\bU^{\sfT}=0$.
In view of the last display we have 
\begin{align}\label{eq:m6}
	\|\mathbb{E} [\bmu_{y_{i}}\bmu_{y_{i}}^{\sfT}]\|=\|\tilde \bLambda\|
	\geq \|(\bI-\bU\bU^{\sfT})\bM\tilde \bLambda\bM^{\sfT}(\bI-\bU\bU^{\sfT})^{\sfT}\| = \|\bar \bLambda\|.
\end{align}
In view of \eqref{eq:m7} we note that 
\begin{equation*}
    \tilde{\bSigma} =\bU \bLambda \bU^{\sfT} + (\bI-\bU\bU^{\sfT})\bM\tilde \bLambda\bM^{\sfT}(\bI-\bU\bU^{\sfT})^{\sfT} = (\bU, \bU_{\perp},\bU_0)\begin{pmatrix}
 \bLambda & &  \\
  & \bar \bLambda & \\
 & & 0
 \end{pmatrix} (\bU, \bU_{\perp},\bU_0)^{\sfT}
\end{equation*}
where $\bU_{0}$ is the orthogonal complement of $(\bU, \bU_{\perp})$. In view of \prettyref{asmp:factors_genr} and \eqref{eq:m6}, as $\bLambda$ is an $r\times r$ matrix we get 
\begin{align}\label{eq:m8}
\lambda_{r}(\bLambda) =\sigma_{\min}^2(\bB) \geq 3\|\bar \bLambda\| 	
\end{align} 
which implies that the set of top $r$ eigenvectors of $\tilde{\bSigma}$ is given by $\bU$. Since $\bV_{r}$ is the top $r$ eigenvectors of $\bSigma$, by \prettyref{lmm:davis-kahan}, 
\begin{align} \label{eq:dist_U_Vr}
\dist (\bU, \bV_{r}) := \|\bU\bU^{\sfT}- \bV_{r}\bV_{r}^{\sfT}\| \leq \frac{2\|\bSigma- \tilde{\bSigma}\|}{\lambda_{r}(\tilde{\bSigma})- \lambda_{r+1}(\tilde{\bSigma})}.
\end{align}
We bound the numerator and denominator of the right most term in the above display in the following way. To bound the denominator, we use \eqref{eq:m8} to get via \prettyref{lmm:davis-kahan}
\begin{align}\label{eq:m9}
	\lambda_{r}(\tilde{\bSigma}) - \lambda_{r+1}(\tilde{\bSigma})\geq \sigma_{\min}(\bB)^{2} - \|\bar \bLambda\| \geq \frac12 \sigma_{\min}(\bB)^{2}
\end{align} 
To bound the numerator $\|\bSigma- \tilde{\bSigma}\|$ in \eqref{eq:dist_U_Vr} we note the decomposition
\begin{align}\label{eq:m10}
    \bSigma- \tilde{\bSigma} & = (\bU \bLambda \bU^{\sfT} + \bM\bV\bM^{\sfT} + \EE [\bepsilon_{i}\bepsilon_{i}^{\sfT}]) - (\bU \bLambda \bU^{\sfT} + (\bI-\bU\bU^{\sfT})\bM\bV\bM^{\sfT}(\bI-\bU\bU^{\sfT})^{\sfT}) 
    \nonumber\\
    &= \bU\bU^{\sfT} \bM\bV\bM^{\sfT} +\bM\bV\bM^{\sfT}\bU\bU^{\sfT} - \bU\bU^{\sfT} \bM\bV\bM^{\sfT} \bU\bU^{\sfT} + \EE [\bepsilon_{i}\bepsilon_{i}^{\sfT}].
\end{align}
Recall $\bepsilon_i \sim \subG_{d} (\sigma^2)$, which means for any $\bv \in \RR^d$ such that $\|\bv\|=1$, $\bv^{\sfT}\bepsilon_{i} \sim \subG (\sigma^2)$. Thus,
\begin{equation*}
    \bv^{\sfT} \EE [\bepsilon_{i}\bepsilon_{i}^{\sfT}] \bv = \EE \qth{(\bv^{\sfT}\bepsilon_{i})^2} \leq \tilde C_1 \sigma^2,
\end{equation*}
for any $\|v\|=1$, where $\tilde C_1>0$ is a constant. Therefore $\|\EE [\bepsilon_{i}\bepsilon_{i}^{\sfT}]\| \leq \tilde C_1\sigma^2$. We then continue \eqref{eq:m10} to have
\begin{align*}
    \|\bSigma- \tilde{\bSigma} \|& \leq  \|\bU\bU^{\sfT} \bM\bV\bM^{\sfT}\| +\|\bM\bV\bM^{\sfT}\bU\bU^{\sfT}\| + \|\bU\bU^{\sfT} \bM\bV\bM^{\sfT} \bU\bU^{\sfT}\| + \|\EE [\bepsilon_{i}\bepsilon_{i}^{\sfT}]\|  \\
    &\leq 2\|\bU\|\|\bU^{\sfT}\bM\|\|\bV\|\|\bM^{\sfT}\| + \|\bU\|\|\bU^{\sfT} \bM\|\|\bV\|\bM^{\sfT} \bU\|\|\bU^{\sfT}\| +c\sigma^2 \\
    & \leq \tilde C_2(\|\bU^{\sfT}\bM\| + \sigma^2)
\end{align*}
for a constant $\tilde C_2>0$. Using \eqref{eq:dist_U_Vr}, 
\begin{align}\label{eq:m11}
	\dist (\bU, \bV_{r})  \leq C\pth{\frac{\|\bU^{\sfT}\bM\| + \sigma^2}{\sigma_{\min}(\bB)^2}} 
	\leq \tilde C_3\pth{\frac{1}{\sigma_{\max}(\bB)} \pth{\sigma \vee \sqrt{\frac{1}{\log n}}}},
\end{align}
for some constant $\tilde C_3>0$, where the last inequality follows from \prettyref{asmp:factors_genr}(c). 
We are now in a position to proceed with the proof of our desired results. 

\begin{itemize}
	\item To show the first claim in \prettyref{lmm:results_on_mean_matrices}, we note
\begin{align*}
    \|(\bI-\bV_{r}\bV_{r}^{\sfT})\bB\| &\stepa{\leq} \|(\bI-\bU\bU^{\sfT})\bB\|  +  \|(\bU\bU^{\sfT}- \bV_{r}\bV_{r}^{\sfT})\bB\| \\
    &\stepb{=} 0 + \|(\bU\bU^{\sfT}- \bV_{r}\bV_{r}^{\sfT})\bB\| \\
    &\stepc{\leq}  \dist (\bU,\bV_{r})\|\bB\|  \\
    &\stepd{\leq}  \frac{\tilde C_4}{\sigma_{\max}(\bB)} \pth{\sigma \vee \sqrt{\frac{1}{\log n}}} \times \sigma_{\max}(\bB)
    \leq \tilde C_4\pth{\sigma \vee \sqrt{\frac{1}{\log n}}},
\end{align*}
where (a) follows from triangle inequality, (b) follows from the spectral decomposition $\bB\bB^\sfT=\bU\bLambda\bU^\sfT$, (c) follows from $\|(\bU\bU^{\sfT}- \bV_{r}\bV_{r}^{\sfT})\bB\|\leq \|(\bU\bU^{\sfT}- \bV_{r}\bV_{r}^{\sfT})\|\|\bB\|$ and $\|\bU\bU^{\sfT}- \bV_{r}\bV_{r}^{\sfT}\|=\dist(\bU,\bV_r)$ as defined in \eqref{eq:dist_U_Vr} and (d) follows from \eqref{eq:m11}.

\item For the third claim, we note using triangle inequality and \eqref{eq:dist_U_Vr} that for any $j_1\neq j_2$
\begin{align*}
    &\| (\bI-\bV_{r}\bV_{r}^{\sfT})(\bmu_{j_1} -\bmu_{j_2}) \|  \\
    &\geq \|(\bI-\bU\bU^{\sfT})(\bmu_{j_1} -\bmu_{j_2}) \| - \|(\bU\bU^{\sfT}- \bV_{r}\bV_{r}^{\sfT})(\bmu_{j_1} -\bmu_{j_2})\| \\
    &\geq \|\bmu_{j_1} -\bmu_{j_2}\| - \|\bU\bU^{\sfT}(\bmu_{j_1} -\bmu_{j_2})\| - \dist (\bU,\bV_{r})\|\bmu_{j_1} -\bmu_{j_2}\|.
\end{align*}
Notice that $\bmu_{j_1} -\bmu_{j_2} \in \text{Col} (\bM)$. As $\bM$ is an orthogonal matrix, there exists $\balpha$ such that $\bmu_{j_1} -\bmu_{j_2} = \bM \balpha$, and $\|\bmu_{j_1} -\bmu_{j_2}\| = \|\balpha\|$. Then we have 
\begin{align*}
   \|\bU\bU^{\sfT}(\bmu_{j_1} -\bmu_{j_2})\|  = \|\bU\bU^{\sfT}\bM \balpha \| 
   \leq \|\bU\|\|\bU^{\sfT}\bM\|\| \balpha\| 
    \leq \frac{1}{2} \|\balpha\| = \frac{1}{2} \|\bmu_{j_1} -\bmu_{j_2}\|,
\end{align*}
where the last inequality followed using $\|\bU^\sfT \bM\|\leq \gamma_5$ from \prettyref{asmp:factors_genr} and we can pick $\gamma_5>0$ to be very small. Therefore, using \eqref{eq:m11}, as $\sigma<c$ for a sufficiently small constant $c>0$ (\prettyref{asmp:model}), we get
\begin{align*}
    &\| (\bI-\bV_{r}\bV_{r}^{\sfT})(\bmu_{j_1} -\bmu_{j_2}) \| \\
    &\geq \|\bmu_{j_1} -\bmu_{j_2}\| - \|\bU\bU^{\sfT}(\bmu_{j_1} -\bmu_{j_2})\| - \dist (\bU,\bV_{r})\|\bmu_{j_1} -\bmu_{j_2}\| \\
    & \geq \|\bmu_{j_1} -\bmu_{j_2}\|- \frac{1}{2} \|\bmu_{j_1} -\bmu_{j_2}\| - \frac 14 \|\bmu_{j_1} -\bmu_{j_2}\|  
    \geq \frac 14 \|\bmu_{j_1} -\bmu_{j_2}\|.
\end{align*}

\item For the second claim, notice that by \prettyref{lmm:weyl} we have,
\begin{align}\label{eq:m12}
     &\lambda_{k} ((\bI-\bV_{r}\bV_{r}^{\sfT})\mathbb{E}[\bmu_{y_{i}}\bmu_{y_{i}}^{\sfT}](\bI-\bV_{r}\bV_{r}^{\sfT})) 
     \nonumber \\ 
     &\geq  \lambda_{k} ((\bI-\bU\bU^{\sfT})\mathbb{E}[\bmu_{y_{i}}\bmu_{y_{i}}^{\sfT}](\bI-\bU\bU^{\sfT}))
     \nonumber\\
     &\quad 
     - \|(\bI-\bU\bU^{\sfT})\mathbb{E}[\bmu_{y_{i}}\bmu_{y_{i}}^{\sfT}](\bI-\bU\bU^{\sfT})- (\bI-\bV_{r}\bV_{r}^{\sfT})\mathbb{E}[\bmu_{y_{i}}\bmu_{y_{i}}^{\sfT}](\bI-\bV_{r}\bV_{r}^{\sfT})\| 
     \nonumber \\ 
     & \geq \lambda_{k} (\mathbb{E}[\bmu_{y_{i}}\bmu_{y_{i}}^{\sfT}]) - \| \mathbb{E}[\bmu_{y_{i}}\bmu_{y_{i}}^{\sfT}] -
     (\bI-\bU\bU^{\sfT})\mathbb{E}[\bmu_{y_{i}}\bmu_{y_{i}}^{\sfT}](\bI-\bU\bU^{\sfT}) \| 
     \nonumber\\
     &
     \quad- \|(\bI-\bU\bU^{\sfT})\mathbb{E}[\bmu_{y_{i}}\bmu_{y_{i}}^{\sfT}](\bI-\bU\bU^{\sfT}) - (\bI-\bV_{r}\bV_{r}^{\sfT})\mathbb{E}[\bmu_{y_{i}}\bmu_{y_{i}}^{\sfT}](\bI-\bV_{r}\bV_{r}^{\sfT})\|.
\end{align}
We also have for a constant $\tilde C_5$
\begin{align}\label{eq:m13}
    &\| \mathbb{E}[\bmu_{y_{i}}\bmu_{y_{i}}^{\sfT}] -
     (\bI-\bU\bU^{\sfT})\mathbb{E}[\bmu_{y_{i}}\bmu_{y_{i}}^{\sfT}](\bI-\bU\bU^{\sfT}) \| 
     \nonumber \\
    & = \|\bU\bU^{\sfT} \bM\tilde \bLambda\bM^{\sfT} +\bM\tilde \bLambda\bM^{\sfT}\bU\bU^{\sfT} - \bU\bU^{\sfT} \bM\tilde \bLambda\bM^{\sfT} \bU\bU^{\sfT}\| 
    \nonumber \\
    & \leq 2 \|\bU^{\sfT}\bM\|\|\tilde \bLambda\| + \|\bU^{\sfT}\bM\|^{2}\|\tilde \bLambda\| 
    \stepa{\leq} \tilde C_5 (2\|\bU^{\sfT}\bM\| + \|\bU^{\sfT}\bM\|^{2})
    \stepb{\leq} \gamma,
\end{align}
where (a) followed from \prettyref{asmp:regularity} and (b) holds a very small constant $\gamma>0$ using \prettyref{asmp:factors_genr}.
We also have for a sufficiently small constant $\tilde\gamma>0$
\begin{align}\label{eq:m14}
    &\|(\bI-\bU\bU^{\sfT})\mathbb{E}[\bmu_{y_{i}}\bmu_{y_{i}}^{\sfT}](\bI-\bU\bU^{\sfT}) - (\bI-\bV_{r}\bV_{r}^{\sfT})\mathbb{E}[\bmu_{y_{i}}\bmu_{y_{i}}^{\sfT}](\bI-\bV_{r}\bV_{r}^{\sfT})\| 
    \nonumber \\
    & = \|(\bU\bU^{\sfT}\bM\tilde\bLambda\bM^{\sfT}\bU\bU^{\sfT} - \bU\bU^{\sfT}\bM\tilde\bLambda\bM^{\sfT} - \bM\tilde\bLambda\bM^{\sfT}\bU\bU^{\sfT}) 
    \nonumber \\
    &- (\bV_{r}\bV_{r}^{\sfT}\bM\tilde\bLambda\bM^{\sfT}\bV_{r}\bV_{r}^{\sfT}  - \bV_{r}\bV_{r}^{\sfT}\bM\tilde\bLambda\bM^{\sfT} - \bM\tilde\bLambda\bM^{\sfT}\bV_{r}\bV_{r}^{\sfT})\|  
    \nonumber \\
    & \leq \|\bU\bU^{\sfT}\bM\tilde\bLambda\bM^{\sfT}\bU\bU^{\sfT} - \bV_{r}\bV_{r}^{\sfT}\bM\tilde\bLambda\bM^{\sfT}\bV_{r}\bV_{r}^{\sfT}\| 
    \nonumber \\
    &+ 2\|\bU\bU^{\sfT}\bM\tilde\bLambda\bM^{\sfT} - \bV_{r}\bV_{r}^{\sfT}\bM\tilde\bLambda\bM^{\sfT}\| 
    \nonumber \\
    & \leq 2 \|\tilde\bLambda\| \|\bU\bU^{\sfT} - \bV_{r}\bV_{r}^{\sfT} \| + 2 \|\tilde\bLambda\| \|\bU\bU^{\sfT} - \bV_{r}\bV_{r}^{\sfT} \| 
    \leq \tilde \gamma
\end{align}
where the last inequality followed using \eqref{eq:m10} and \eqref{eq:m11} with $\sigma<\gamma_5$ from \prettyref{asmp:model} and $\|\tilde \bLambda\|=\lambda_1(\EE[\bmu_{y_i}\bmu_{y_i}^\sfT])\leq C_3$ from \prettyref{asmp:regularity}.
Combining \eqref{eq:m12},\eqref{eq:m13},\eqref{eq:m14} we get
\begin{align*}
\lambda_{k} ((\bI-\bV_{r}\bV_{r}^{\sfT})\mathbb{E}[\bmu_{y_{i}}\bmu_{y_{i}}^{\sfT}](\bI-\bV_{r}\bV_{r}^{\sfT})) 
\geq \lambda_{k} (\mathbb{E}[\bmu_{y_{i}}\bmu_{y_{i}}^{\sfT}]) - \gamma - \tilde \gamma 
\geq 
\xi_4  
\end{align*}
for a constant $\xi_4>0$. Therefore we proved \prettyref{lmm:results_on_mean_matrices}.

\end{itemize}
\end{proof}

\subsubsection{Isotropic noise}
\begin{proof}
For the case that the noise is isotropic., i.e., $\EE \qth{\bepsilon_i \bepsilon_i^{\sfT}} = \sigma^2 \bI_d$, we will show that \eqref{eq:m11} still holds under the weaker version of \prettyref{asmp:factors_genr}(d). Therefore the proof still works. For showing \eqref{eq:m11} still holds, we let $\bar{\bSigma}$ denote the matrix $\tilde{\bSigma} + \sigma^2 \bI_d$. The same as the non-isotropic case, we will use the decomposition  
\begin{align*}
    \bSigma 
    &= \bU \bLambda \bU^{\sfT} + \bM\tilde \bLambda\bM^{\sfT} + \EE [\bepsilon_{i}\bepsilon_{i}^{\sfT}] \\
    &= \bU \bLambda \bU^{\sfT} + \bM\tilde \bLambda\bM^{\sfT} + \sigma^2 \bI_d,
\end{align*}
and 
\begin{equation*}
    \bar{\bSigma} = (\bU, \bU_{\perp},\bU_0)\begin{pmatrix}
 \bLambda & &  \\
  & \bar \bLambda & \\
 & & 0
 \end{pmatrix} (\bU, \bU_{\perp},\bU_0)^{\sfT} + \sigma^2 \bI_d.
\end{equation*} 
Recall in the proof for non-isotropic case, we show that the set of top $r$ eigenvectors of $\tilde{\bSigma}$ is given by $\bU$. Therefore $\bar{\bSigma} = \tilde{\bSigma} +\sigma^2 \bI_d$  have the same set of top $r$ eigenvectors as $\tilde{\bSigma}$, which is $\bU$ (here we use $\sigma^2 \bI_d$ is isotropic). Since $\bV_{r}$ is the top $r$ eigenvectors of $\bSigma$, by \prettyref{lmm:davis-kahan}, 
\begin{align} \label{eq:dist_U_Vr_2}
\dist (\bU, \bV_{r}) := \|\bU\bU^{\sfT}- \bV_{r}\bV_{r}^{\sfT}\| \leq \frac{2\|\bSigma- \bar{\bSigma}\|}{\lambda_{r}(\bar{\bSigma})- \lambda_{r+1}(\bar{\bSigma})}.
\end{align}
We bound the numerator and denominator of the right most term in the above display in the following way. To bound the denominator, we use \eqref{eq:m8} to get  \begin{align}
	\lambda_{r}(\bar{\bSigma}) - \lambda_{r+1}(\bar{\bSigma}) = \lambda_{r}(\tilde{\bSigma}) - \lambda_{r+1}(\tilde{\bSigma})\geq \sigma_{\min}(\bB)^{2} - \|\bar \bLambda\| \geq \frac12 \sigma_{\min}(\bB)^{2}
\end{align} 
To bound the numerator $\|\bSigma- \bar{\bSigma}\|$ in \eqref{eq:dist_U_Vr_2} we note the decomposition
\begin{align}\label{eq:m15}
    \bSigma- \bar{\bSigma} & = (\bU \bLambda \bU^{\sfT} + \bM\bV\bM^{\sfT} + \sigma^2 \bI_d) - (\bU \bLambda \bU^{\sfT} + (\bI-\bU\bU^{\sfT})\bM\bV\bM^{\sfT}(\bI-\bU\bU^{\sfT})^{\sfT} + \sigma^2 \bI_d) 
    \nonumber\\
    &= \bU\bU^{\sfT} \bM\bV\bM^{\sfT} +\bM\bV\bM^{\sfT}\bU\bU^{\sfT} - \bU\bU^{\sfT} \bM\bV\bM^{\sfT} \bU\bU^{\sfT} \nonumber\\
    & \leq  \|\bU\bU^{\sfT} \bM\bV\bM^{\sfT}\| +\|\bM\bV\bM^{\sfT}\bU\bU^{\sfT}\| + \|\bU\bU^{\sfT} \bM\bV\bM^{\sfT} \bU\bU^{\sfT}\|   \nonumber\\
    &\leq 2\|\bU\|\|\bU^{\sfT}\bM\|\|\bV\|\|\bM^{\sfT}\| + \|\bU\|\|\bU^{\sfT} \bM\|\|\bV\|\bM^{\sfT} \bU\|\|\bU^{\sfT}\|  \nonumber\\
    & \leq \tilde C_1(\|\bU^{\sfT}\bM\|)
\end{align}
for a constant $\tilde C_1>0$. Using \eqref{eq:dist_U_Vr_2}, 
\begin{align}
	\dist (\bU, \bV_{r})  \leq \tilde C_2\pth{\frac{\|\bU^{\sfT}\bM\|}{\sigma_{\min}(\bB)^2}} 
	\leq \tilde C_3\pth{\frac{1}{\sigma_{\max}(\bB)} \pth{\sigma \vee \sqrt{\frac{1}{\log n}}}},
\end{align}
for constants $\tilde C_2,\tilde C_3>0$, where the last inequality follows from the weaker version of \prettyref{asmp:factors_genr}(c). Therefore we show that \eqref{eq:m11} still holds for isotropic noise with weaker assumption. The rest of the proofs remain the same as the non isotropic case.

\end{proof}

\subsection{Proof of \prettyref{lmm:eigenvalue_concentration}}
Back to the proof of \prettyref{lmm:eigenvalue_concentration}. For the first inequality in the lemma statement,
\begin{align*}
	&|\lambda_{k}(\EE[\tilde{\bmu}_{y_i}\tilde{\bmu}_{y_i}^{\sfT}]) - \lambda_{k}(\frac1n\sum_{i=1}^{n}\tilde{\bmu}_{y_i} \tilde{\bmu}_{y_i}^{\sfT})|
	\stepa{\leq} \|\EE[\tilde{\bmu}_{y_i}\tilde{\bmu}_{y_i}^{\sfT}] - \frac1n\sum_{i=1}^{n}\tilde{\bmu}_{y_i} \tilde{\bmu}_{y_i}^{\sfT}\| \\
	&= \|\sum_{j=1}
	^{K}p_{j}\tilde{\bmu}_{j}\tilde{\bmu}_{j}^{\sfT} - \sum_{j=1}
	^{K}\frac{n_j}{n}\tilde{\bmu}_{j}\tilde{\bmu}_{j}^{\sfT} \|
	\leq \sum_{j=1}^{K}|\frac{n_{j}}{n} -p_{j}| \|\tilde{\bmu}_{j}\tilde{\bmu}_{j}^{\sfT}\|
	\stepb{\leq} \tilde C_1\frac{\log n}{\sqrt{n}}
\end{align*}
with probability greater than $1-\tilde C_1/n$, for a constant $\tilde C_1>0$, where $n_j = |\{i \in [n]: y_i=j\}|$. (a) follows from Weyl's inequality (\prettyref{lmm:weyl}), (b) holds since $n_{j}$ is a binomial random variable (so that it has sub-Gaussian tail bound), and $\|\tilde{\bmu}_{j}\tilde{\bmu}_{j}^{\sfT}\|\leq \|\tilde{\bmu}_j\|^2 \leq \|\bmu_j\|^2 \leq \tilde C_2^2$ for a constant $\tilde C_2>0$ by (\ref{eq:bound_mu}). To show the second inequality in the lemma statement we note that (as we have already conditioned for $\hat \bV_r$)
\begin{align*}
	& \lambda_{k} (\mathbb{E}[\tilde{\bmu}_{y_{i}}\tilde{\bmu}_{y_{i}}^{\sfT}]) \\
	&=\lambda_{k} ((\bI-\hat{\bV}_{r}\hat{\bV}_{r}^{\sfT})\mathbb{E}[\bmu_{y_{i}}\bmu_{y_{i}}^{\sfT}](\bI-\hat \bV_{r}\hat \bV_{r}^{\sfT}))\\  
	&\stepa{\geq}  \lambda_{k} ((\bI-\bV_{r}\bV_{r}^{\sfT})\mathbb{E}[\bmu_{y_{i}}\bmu_{y_{i}}^{\sfT}](\bI-\bV_{r}\bV_{r}^{\sfT})) \\
	& - \|(\bI-\bV_{r}\bV_{r}^{\sfT})\mathbb{E}[\bmu_{y_{i}}\bmu_{y_{i}}^{\sfT}](\bI-\bV_{r}\bV_{r}^{\sfT}) - (\bI-\hat{\bV}_{r}\hat{\bV}_{r}^{\sfT})\mathbb{E}[\bmu_{y_{i}}\bmu_{y_{i}}^{\sfT}](\bI-\hat{\bV}_{r}\hat{\bV}_{r}^{\sfT})\| \\ 
	& \stepb{\geq} \lambda_{k} ((\bI-\bV_{r}\bV_{r}^{\sfT})\mathbb{E}[\bmu_{y_{i}}\bmu_{y_{i}}^{\sfT}](\bI-\bV_{r}\bV_{r}^{\sfT})) 
	- \xi\|\bDelta\| \\
	& \stepc{\geq} \xi_3 - {\xi_1\over \sigma_{\max}(\bB)}\pth{\sigma\vee \frac 1{\sqrt{\log n}}}
	\geq  \xi_6,
\end{align*}
for constants $\xi,\xi_6>0$, where (a) follows using \prettyref{lmm:weyl}, (b) follows using \eqref{eq:eigen-decomp} as 
\begin{align*}
    &\|(\bI-\hat{\bV}_r\hat{\bV}_r^{\sfT})\mathbb{E}[\bmu_{y_{i}}\bmu_{y_{i}}^{\sfT}](\bI-\hat{\bV}_r\hat{\bV}_r^{\sfT}) - (\bI-\bV_{r}\bV_{r}^{\sfT})\mathbb{E}[\bmu_{y_{i}}\bmu_{y_{i}}^{\sfT}](\bI-\bV_{r}\bV_{r}^{\sfT})\| 
    \nonumber \\
    & = \|(\hat{\bV}_r\hat{\bV}_r^{\sfT}\bM\tilde\bLambda\bM^{\sfT}\hat{\bV}_r\hat{\bV}_r^{\sfT} - \hat{\bV}_r\hat{\bV}_r^{\sfT}\bM\tilde\bLambda\bM^{\sfT} - \bM\tilde\bLambda\bM^{\sfT}\hat{\bV}_r\hat{\bV}_r^{\sfT}) 
    \nonumber \\
    &- (\bV_{r}\bV_{r}^{\sfT}\bM\tilde\bLambda\bM^{\sfT}\bV_{r}\bV_{r}^{\sfT}  - \bV_{r}\bV_{r}^{\sfT}\bM\tilde\bLambda\bM^{\sfT} - \bM\tilde\bLambda\bM^{\sfT}\bV_{r}\bV_{r}^{\sfT})\|  
    \nonumber \\
    & \leq \|\hat{\bV}_r\hat{\bV}_r^{\sfT}\bM\tilde\bLambda\bM^{\sfT}\hat{\bV}_r\hat{\bV}_r^{\sfT} - \bV_{r}\bV_{r}^{\sfT}\bM\tilde\bLambda\bM^{\sfT}\bV_{r}\bV_{r}^{\sfT}\| 
    \nonumber \\
    &+ 2\|\hat{\bV}_r\hat{\bV}_r^{\sfT}\bM\tilde\bLambda\bM^{\sfT} - \bV_{r}\bV_{r}^{\sfT}\bM\tilde\bLambda\bM^{\sfT}\| 
    \nonumber \\
    & \leq 2 \|\tilde\bLambda\| \|\hat{\bV}_r\hat{\bV}_r^{\sfT} - \bV_{r}\bV_{r}^{\sfT} \| + 2 \|\tilde\bLambda\| \|\hat{\bV}_r\hat{\bV}_r^{\sfT} - \bV_{r}\bV_{r}^{\sfT} \| 
    \leq \xi\|\bDelta\|,
\end{align*}
for a constant $\xi>0$ and (c) follows as \prettyref{lmm:bound_on_Delta_norm} and the last inequality follows whenever $\sigma_{\max}(\bB)>\tilde C_3\sigma$ for a large enough constant $\tilde C_3>0$.

\section{Proofs for \prettyref{cor:dispro}}
In order to proof \prettyref{cor:dispro}, we only need to directly check that \prettyref{asmp:factors_genr} is satisfied when \prettyref{asmp:regularity}, \prettyref{asmp:approximate_perpendicularity_weak} and \prettyref{asmp:approximate_perpendicularity_disproportionate}  and $d^3 \leq O ( \frac{n}{(\log n)^{2}})$ hold (notice that since the noise is isotropic, we only need the weaker version of \prettyref{asmp:factors_genr}(c)).

\end{document}